\documentclass{amsart}
\usepackage{amsmath,amssymb,latexsym, amsfonts, amscd, amsthm}
\usepackage{footnote}
\usepackage{enumerate}

\usepackage{bbm}
\usepackage{enumerate}
\usepackage[all]{xy}

\allowdisplaybreaks
\newtheorem{thm}{Theorem}[section]
\newtheorem{pro}[thm]{Proposition}
\newtheorem{lm}[thm]{Lemma}
\newtheorem{cor}[thm]{Corollary}

\newtheorem{hyp}[thm]{Hypotheses}
\numberwithin{equation}{section}

\theoremstyle{remark}

\newtheorem{rem}[thm]{Remark}

\theoremstyle{definition}

\DeclareMathOperator*{\End}{End}

\DeclareMathOperator*{\Irr}{Irr}
\DeclareMathOperator*{\disc}{disc}
\DeclareMathOperator*{\Gal}{Gal}

\DeclareMathOperator*{\ad}{ad}
\DeclareMathOperator*{\scn}{sc}

\DeclareMathOperator*{\der}{der}

\DeclareMathOperator*{\Hom}{Hom}

\newcommand{\s}{\simeq}

\newcommand{\tsigma}{\widetilde{\sigma}}

\newcommand{\ma}{\mathfrak{a}}

\newcommand{\CC}{\mathbb{C}}

\newcommand{\RR}{\mathbb{R}}
\newcommand{\QQ}{\mathbb{Q}}
\newcommand{\ZZ}{\mathbb{Z}}

\newcommand{\ii}{\mathbf{\textit{i}}}

\def\bA{\bold A}
\def\bG{\bold G}
\def\bP{\bold P}
\def\bM{\bold M}
\def\bN{\bold N}
\def\a{\alpha}
\def\b{\beta}

\def\sC{\mathcal C}

\def\ds{\displaystyle}

\title[Invariance of $R$-groups for Classical Groups]{Invariance of $R$-groups between $p$-adic inner forms of quasi-split classical groups} 

\author[Kwangho Choiy and David Goldberg]{Kwangho Choiy and David Goldberg} 
\thanks{This work was partially written while K. Choiy was a postdoctoral fellow at ICERM at Brown University during the Spring 2013 semester program entitled ``Automorphic Forms,
Combinatorial Representation Theory, and Multiple Dirichlet Series.''}
\thanks{D. Goldberg partially supported by Simons Foundation Collaboration Grant 279153.}


\address{Kwangho Choiy\\
Department of Mathematics\\
Oklahoma State University\\
Stillwater, OK 74078-1058\\
U.S.A.}
\email{kwangho.choiy@okstate.edu}

\address{David Goldberg\\
Department of Mathematics \\
Purdue University\\
West Lafayette, IN 47907\\
U.S.A.}
\email{goldberg@math.purdue.edu}
\date{\today}
\begin{document}
\maketitle  

\begin{abstract} 
We study the reducibility of parabolically induced representations of non-split inner forms of quasi-split classical groups.  The isomorphism of Arthur $R$--groups, endoscopic $R$--groups and Knapp-Stein $R$--groups is established, as well as showing these $R$--groups are isomorphic to the corresponding ones for the quasi-split form.  This shows $R$--groups are an invariant of the $L$--packets.    The results are applied to classify the  elliptic spectrum.

\end{abstract}
\section{Introduction}
We continue our study of $L$--packets and reducibility for non quasi-split inner forms of reductive $p$--adic groups.  In \cite{goldberg-choiy-1} we examined the case of inner forms of $SL_n.$  There we found a relation between the structure of the tempered $L$--packets of $SL_n(F)$ and those of $SL_r(D),$ where $D$ is a central division algebra of dimension $d^2$ over $F,$ and $r=n/d.$ Our method was to use the  transfer of Plancherel measures via a Jacquet-Langlands type correspondence developed by the first named author in \cite{choiy}. This allows us to determine the Knapp-Stein $R$--groups for $SL_r(D)$ from those of $SL_n(F).$  The same result was proved by Chao and Li \cite{chao-li} using a different method, and they showed there are $L$--packets of the non quasi-split inner form whose $R$-groups are strictly smaller than those of the corresponding $L$--packet of $SL_n(F).$   Here we turn to the case of the split classical groups $SO_n$ and $Sp_{2n},$ as well as the non-split, quasi-split special orthogonal group $SO_{2n}^*.$  In these cases we can show the $R$--groups are invariant under the transfer to the non quasi-split inner form.  Further, we can show the intertwining algebras are preserved.  This shows the theory of reducibility of induced from discrete series representations, and elliptic representations for these cases transfers from  \cite{goldberg-class,herb93}.

We fix a $p$--adic field, $F,$ of characteristic zero, and  consider a reductive connected quasi-split algebraic group $\bG$ defined over $F.$  We assume $\bG'$ is an inner form of $\bG.$  The tempered spectrum of $G=\bG(F)$ consists of the discrete series, as well as  the irreducible components of representations parabolically induced from discrete series representations of proper $F$--Levi subgroups.  If $\bM'$ is an $F$--Levi subgroup of $\bG',$ then there is an $F$--Levi subgroup $\bM$ of $\bG$ with $\bM'$ an inner form of $\bM.$  Since the root datum of $\bG$ determines the Langlands $L$--group $^LG,$ we have an $\bar{F}$-isomorphism $\psi: \bG' \overset{\sim}{\rightarrow} \bG$ 
and we see $\hat
G=\hat G',$ where these are the connected components of the  $L$--groups.   The local Langlands correspondence gives a parameterization of the tempered spectrum, in the form of $L$--packets, via admissible homomorphisms $\varphi:W_F'\rightarrow\,^LG.$  This, then, gives a correspondence between the $L$--packets $\Pi_\varphi(G)$ of $G$ and $\Pi_\varphi(G')$ of $G'.$  Since $\bM$ and $\bM'$ are inner forms, we have a similar correspondence between $L$--packets on the Levi subgroups.

For the groups under consideration, $SO_n, Sp_{2n}$ and $SO_{2n}^*,$ the theory of induced representations is well understood (see \cite{goldberg-class} and Appendix A).  Furthermore, from the work of Harris-Taylor, \cite{ht01}, Henniart, \cite{he00} and Arthur, \cite{art12}, we know the local Langlands correspondence has been established for all these groups, as well as all of their Levi subgroups (see also \cite{js04, mw03} for $SO_{2n+1}$).  Here we investigate the induced representations for the non quasi-split inner forms $G'$ of these groups $G.$  Suppose $\sigma'$ is a discrete series representation of some Levi subgroup $M'=\bM'(F)$ of $G'.$ 
We suppose $\sigma'\in\Pi_\varphi (M'),$ and take an element $\sigma\in\Pi_\varphi (M).$  
We show there is a one to one correspondence between the components of the induced representations $i_{G,M}(\sigma)$ and $i_{G',M'}(\sigma')$ (cf. Theorem 3.6, and its corollaries).   The structure of $i_{G,M}(\sigma)$ and $i_{G',M'}(\sigma')$ are determined by the intertwining algebras $\mathcal C(\sigma)$ and $\mathcal C(\sigma').$  These in turn are determined by the Knapp-Stein $R$--groups $R_\sigma$ and $R_{\sigma'},$ along with two cocylces of each of these groups.  We show $R_\sigma\simeq R_{\sigma'}.$  Herb showed the cocycle of $R_\sigma$ is trivial, \cite{herb93}, and we show the argument there can be adapted to the inner form, showing the cocylce of $R_{\sigma'}$ is also trivial.  This, then gives the correspondence between components we referred to above.    Furthermore,  we show the components of $i_{G',M'}(\sigma')$ are elliptic if and only if the components of  $i_{G,M}(\sigma)$ are elliptic.

Our approach is through the theory of endoscopic and Arthur $R$--groups.  These are $R$--groups attached to the packet through the  $\hat G$--centralizer of the image of the parameter, $\varphi$ (see sections 2 and 3).  In order to establish our results we need to work under a pair of hypotheses (cf. Hypotheses 3.2).  Namely, we have to assume the stabilization of the twisted trace formula, and also assume the classification for the inner forms, which are expected to follow from the work of Arthur \cite{art12}, though this has not been completed as of now.   In the case of a split group  the endoscopic $R$--group $R_\varphi$ is the quotient of the Weyl group  $W_\varphi$ of the centralizer modulo that of the connected component $W_\varphi^\circ.$ The Arthur $R$-group is then defined by identifying $W_\varphi$ with a subgroup of the Weyl group $W_M.$  Then $R_{\varphi,\sigma}=W_{\varphi,\sigma}/W_{\varphi,\sigma}^\circ,$ where the two factors are determined by intersection with  the $W_M$--stabilizer of $\sigma.$  Then our main result is $R_\sigma\simeq R_{\varphi,\sigma}\simeq R_{\varphi,\sigma'}\simeq R_{\sigma'}$ (cf. Theorem 3.6).  From this all the results on reducibility and components now follow directly.

In Section 2 we give our  definitions and notation.  We give an explicit description of all the inner forms we will study, as well as preliminary results we need.  Section 3 then contains the main results on elliptic tempered $A$--packets, endoscopic, and Arthur $R$--groups. In Section 4, we apply the results of Section 3 to the study of elliptic representations.  Finally, Appendix A contains an extension of the results of \cite{goldberg-class} to the case of the non-split quasi-split even special orthogonal group.

\section{Preliminaries} \label{pre}
\subsection{Basic notation} \label{basic notation}
Let $F$ denote a $p$-adic field of characteristic $0,$ that is, 
a finite extension of $\QQ_p.$ Fix an algebraic closure $\bar{F}$ of $F.$ 
We let $\bG$ denote a connected reductive algebraic group defined over $F.$ 
We use the notation $G$ to denote the group $\bG(F)$ of $F$-points and likewise for other algebraic groups defined over $F.$

Fix a minimal $F$-parabolic subgroup $\bP_0$ of $G$ with Levi component $\bM_0$ and unipotent radical $\bN_0.$ Let $\bA_0$ be the split component of $\bM_0,$ that is, the maximal $F$-split torus in the center of $\bM_0.$ Let $\Delta$ be the set of simple roots of $\bA_0$ in $\bN_0.$ Let $\bP \subseteq \bG$ be a standard (that is, containing $\bP_0$) $F$-parabolic subgroup of $\bG.$ 
Write $\bP=\bM\bN$ with its Levi component $\bM=\bM_{\Theta} \supseteq \bM_0$ generated by a subset $\Theta \subseteq \Delta$ and its unipotent radical $\bN \subseteq \bN_0.$
We note that the split component $\bA_{\bM}=\bA_{\bM_{\Theta}}$ of $\bM=\bM_{\Theta}$ is equal to the identity component
$(\cap_{\alpha \in \Theta} \ker \alpha
)^{\circ}$ in $\bA_0,$
so that $\bM = Z_\bG(\bA_\bM),$ 
where $Z_\bG(\bA_\bM)$ is the centralizer of $\bA_\bM$ in $\bG$ (see \cite[Proposition 20.4]{borel91} and \cite[Section 15.1]{springer98}).
We denote by $\delta_P$ the modulus character of $P.$ 

Let $\Irr(M)$ denote the set of isomorphism classes of irreducible admissible representations of $M.$ Denote by $\Pi_{\disc}(M)$ the set of discrete series representations of $M.$ By abuse of terminology, we do not make distinction between each isomorphism class and its representative. For any $\sigma \in \Irr(M),$ we write $\ii_{G,M} (\sigma)$ for the normalized (twisted by $\delta_{P}^{1/2}$) induced representation. We denote by $\tsigma$ the contragredient of $\sigma.$

We let $X^{*}(\bM)_F$ denote the group of $F$-rational characters of $\bM.$ We set
$\ma_M := \Hom(X^{*}(\bM)_F ,\RR)$
the real Lie algebra of $\bA_{\bM}.$ 
We denote by $\Phi(P, A_M)$ the set of reduced roots of $\bP$ with respect to $\bA_\bM.$ 
Let $W_M = W(\bG, \bA_\bM) := N_\bG(\bA_\bM) / Z_\bG(\bA_\bM)$ denote the Weyl group of $\bA_\bM$ in $\bG,$ 
where $N_\bG(\bA_\bM)$ is the normalizer of $\bA_\bM$ in $\bG.$ 
For simplicity, we write $\bA_0 = \bA_{\bM_0}.$ 

We denote by $W_F$ the Weil group of $F$ and $\Gamma := \Gal(\bar{F} / F).$ Fixing $\Gamma$-invariant splitting data, we define $L$-group $^{L}G$ of $\bG$ as a semi-direct product $^{L}G := \widehat{G} \rtimes \Gamma$ (see \cite[Section 2]{bo79}).
For any topological group $H,$ we denote by $Z(H)$ the center of $H.$ We let $\pi_0(H)$ denote the group $H/H^\circ$ of connected components of $H,$ where $H^\circ$ is the identity component of $H.$
\subsection{Structure of $F$-inner forms of quasi-split classical groups} \label{structure}
Let $\bG=\bG(n)$ denote a quasi-split classical group of rank $n$ over $F.$ 
More precisely, they are the following cases.
\begin{enumerate}
\item Type $\bold{B_n}:$ $\bG(n)=SO_{2n+1},$ the split special orthogonal group in $2n+1$ variables defined over $F.$
\item Type $\bold{C_n}:$ $\bG(n)=Sp_{2n},$ the symplectic group in $2n$ variables defined over $F.$
\item Type $\bold{{^1}D_n}:$ $\bG(n)=SO_{2n},$ the split special orthogonal group in $2n$ variables defined over $F.$ 
\item Type $\bold{{^2}D_n}:$ $\bG(n)=SO^*_{2n},$ the quasi-split special orthogonal group associated to a quadratic extension $E$ over $F,$ where $\bG(n)$ is split over $E.$
\end{enumerate}
We denote by $\bM$ an $F$--Levi subgroup of $\bG.$ Then, $M$ is of the form
\begin{equation} \label{form of M}
GL_{n_1}(F) \times \cdots \times GL_{n_k}(F) \times G_-(m),
\end{equation}
where $\sum_{i=1}^k n_i + m = n$ and $\bG_-(m)$ denotes $SO_{2m+1}, Sp_{2m}, SO_{2m}$ or $SO^*_{2m}$ with the same type as $\bG.$ 
We refer the reader to \cite{goldberg-class} for split cases and Appendix \ref{R-gps for SO*} for the quasi-split case $SO^*_{2n.}$

We let $\bG'=\bG'(n)$ denote an $F$--inner form of $\bG.$
Recall that $\bG$ and $\bG'$ are \textit{$F$-inner forms} with respect to an $\bar{F}$-isomorphism $\varphi: \bG' \overset{\sim}{\rightarrow} \bG$ 
if $\varphi \circ \tau(\varphi)^{-1}$ is an inner automorphism ($g \mapsto xgx^{-1}$) defined over $\bar{F}$ for all $\tau \in \Gamma.$ 
We often omit $\varphi$ with no danger of confusion.
Note that $\bG'$ can be $\bG$ itself by taking $\varphi$ to be the identity map.
We denote by $\bM'$ an $F$--Levi subgroup of $\bG'$ such that $\bM'$ is an $F$--inner form of $\bM.$

For the rest of the section, we discuss the structure of $M',$ which turns out to be of the form
\begin{equation} \label{form of M'}
GL_{m_1}(D) \times \cdots \times GL_{m_k}(D) \times G'_-(m).
\end{equation}
Here $D$ denotes a central division algebra of dimension $1$ (hence, $D=F$) or $4$ over $F,$ and $\bG'_-(m)$ denotes an $F$-inner form of $SO_{2m+1}, Sp_{2m}, SO_{2m}$ or $SO^*_{2m}$ with the same type as $\bG'.$ 
Further, $ \sum_{i=1}^k d ~ m_i + m = n,$ where $d=1$ or $2$ according as the dimension of $D$ which is $1$ or $4.$ Note that $d ~ m_i = n_i.$

In what follows, based on the Satake classification \cite[pp. 119-120]{sa71}, we describe every $F$-inner form $\bG'$ and every possible maximal $F$-Levi subgroup $\bM'$ of $\bG'.$
In the diagram below (Satake diagram), a black vertex indicates a root in the set of simple roots of a fixed minimal $F$-Levi subgroup $\bM'_0$ of $\bG'$. 
So, we remove only a subset $\vartheta$ of white vertices to obtain an $F$-Levi subgroup $\bM'$ (see \cite[Section 2.2]{sa71} and \cite[Section I.3]{bo79}). 
As discussed in Section \ref{basic notation}, the $F$-Levi subgroup $\bM',$ corresponding to $\Theta = \Delta \setminus \vartheta,$  
is the centralizer in $\bG'$ of the split component $\bA_{\bM'}=(\cap_{\alpha \in \Theta} \ker \alpha)^{\circ}.$ 
One thus notices that there is an $F$--isomorphism between two split components
\begin{equation}  \label{iso between A's}
\bA_{\bM} \s \bA_{\bM'} \s (GL_1)^k.
\end{equation}
We write $D_2$ and $D_4$ for a division algebra of dimension $4$ and $16$ over $F,$ respectively. 
\bigskip

\noindent \textbf{$\bold{B_n}$-type :} There is only one (up to isomorphism) non-split $F$-inner form $\bG'=\bG'(n)$ of $SO_{2n+1}$ with the following diagram.
\begin{itemize}
\item[]

\[
\xy 
\POS (10,0) *\cir<2pt>{} ="a" ,
\POS (20,0) *\cir<2pt>{} ="b" ,
\POS (30,0) *\cir<0pt>{} ="c" ,
\POS (40,0) *\cir<0pt>{} ="d" ,
\POS (50,0) *\cir<2pt>{} ="e" ,
\POS (60,0) *\cir<2pt>{} ="f",
\POS (70,0) *{\bullet} ="g"

\POS "a" \ar@{-}^<<<{}_<<{}  "b",
\POS "b" \ar@{-}^<<<{}_<<{}  "c",
\POS "c" \ar@{.}^<<<{}_<<{}  "d",
\POS "d" \ar@{-}^<<<{}_<<{}  "e",
\POS "e" \ar@{-}^<<<{}_<<{}  "f",
\POS "f" \ar@{=>}^<<<{}_<<{} "g",
\endxy
\]
Set $\Theta = \Delta \setminus  \alpha_{i} $, where $i=1, 2, \dots, n-1$ and $\alpha_{i} = e_{i}-e_{i+1}.$ Then $M'=M'_{\Theta}$ is of the form $GL_{i}(F) \times G'(n-i).$ For the case $i=n-1,$ $G'(n-i)=G'(1)=PSL_1(D_2).$  

\end{itemize}

\bigskip

\noindent \textbf{$\bold{C_n}$-type :} According to the parity of $n,$ there is only one (up to isomorphism) non-split $F$-inner form $\bG'=\bG'(n)$ of $Sp_{2n}$ with the following diagram.

\begin{itemize}
\item $\mathbf{n}$\textbf{: odd}
\[
\xy 
\POS (10,0) *{\bullet} ="a" ,
\POS (20,0) *\cir<2pt>{} ="b" ,
\POS (30,0) *{\bullet} ="c" ,
\POS (40,0) *\cir<2pt>{} ="d" ,
\POS (50,0) *\cir<0pt>{} ="e" ,
\POS (60,0) *\cir<0pt>{} ="f",
\POS (70,0) *{\bullet} ="g",
\POS (80,0) *\cir<2pt>{} ="h",
\POS (90,0) *{\bullet} ="i",

\POS "a" \ar@{-}^<<<{}_<<{}  "b",
\POS "b" \ar@{-}^<<<{}_<<{}  "c",
\POS "c" \ar@{-}^<<<{}_<<{}  "d",
\POS "d" \ar@{-}^<<<{}_<<{}  "e",
\POS "e" \ar@{.}^<<<{}_<<{}  "f",
\POS "f" \ar@{-}^<{}_<<<{} "g",
\POS "g" \ar@{-}^<{}_<<<{} "h",
\POS "h" \ar@{<=}^<<<<{} "i",
\endxy
\]
\begin{center} 
(every other dot black)
\end{center}
Set $\Theta = \Delta \setminus \alpha_{i}$, where $i=2, 4, \dots, n-1$ and $\alpha_{i}=e_{i}-e_{i+1}.$ Then $M'=M'_{\Theta}$ is of the form $GL_{i/2}(D_2) \times G'(n-i).$ For the case $i=n-1,$ $G'(n-i)=G'(1)=SL_1(D_2).$

\item $\mathbf{n}$\textbf{: even}

\[
\xy 
\POS (10,0) *{\bullet} ="a" ,
\POS (20,0) *\cir<2pt>{} ="b" ,
\POS (30,0) *{\bullet} ="c" ,
\POS (40,0) *\cir<2pt>{} ="d" ,
\POS (50,0) *\cir<0pt>{} ="e",
\POS (60,0) *\cir<0pt>{} ="f",
\POS (70,0) *{\bullet} ="g",
\POS (80,0) *\cir<2pt>{} ="h",
\POS (90,0) *{\bullet} ="i",
\POS (100,0) *\cir<2pt>{} ="j",

\POS "a" \ar@{-}^<<<{}_<<{}  "b",
\POS "b" \ar@{-}^<<<{}_<<{}  "c",
\POS "c" \ar@{-}^<<<{}_<<{}  "d",
\POS "d" \ar@{-}^<<<{}_<<{}  "e",
\POS "e" \ar@{.}^<{}_<<<{} "f",
\POS "f" \ar@{-}^<{}_<<<{} "g",
\POS "g" \ar@{-}^<{}_<<<{} "h",
\POS "h" \ar@{-}^<{}_<<<{} "i",
\POS "i" \ar@{<=}^>>>>{} "j",

\endxy
\]
\begin{center} 
(every other dot black)
\end{center}
Set $\Theta = \Delta \setminus \alpha_i $, where $i=2, 4, \dots, n-2, n$ and $\alpha_{i}= e_{i}-e_{i+1}$ for $i \neq n;$ $\alpha_{n}=2e_{n}$ for $i = n.$ 
Then $M'=M'_{\Theta}$ is of the form $GL_{i/2}(D_2) \times SU^{+}_{n-i}(D_2).$ Here, $SU^{+}_{2k}(D_2)$ is defined as
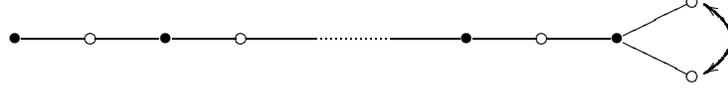
\[
SU^{+}_{2k}(D_2) := \{ g \in GL_{2k}(D_2) :  g^{*} J^+ g = J^+ \},
\]
where $g^{*} = (g_{ij})^{*} = ( \bar{g}_{ji})$  with the usual involution $\bar{g}$ on $D$,  $J^+ = \bigl(\begin{smallmatrix} 0 & I_k \\  I_k & 0 \end{smallmatrix} \bigr),$ 
and $SU^{+}_{0}(D_2)=G'(0)=1$ (see \cite[Section 2.3.3]{pr94}). For the case $i=n,$ $M'=M'_{\Theta}$ is of the form $GL_{n/2}(D_2) .$ We remark that, for even $n$, the $F$-inner form $G'=G'(n)$ of $Sp_{2n}$  is of the form $SU^{+}_{n}(D_2) .$ 
\end{itemize}

\bigskip

\noindent \textbf{$\bold{{^1}D_n}$-type :} For each $n,$ there are two (up to isomorphism) non-split $F$-inner forms $\bG'=\bG'(n)$ of  a split group $SO_{2n}$ with the following diagrams.

\begin{itemize}
\item $\mathbf{n}$\textbf{: any}

\[
\xy
\POS (10,0) *\cir<2pt>{} ="a" ,
\POS (20,0) *\cir<2pt>{} ="b" ,
\POS (30,0) *\cir<0pt>{} ="c" ,
\POS (40,0) *\cir<0pt>{} ="d" ,
\POS (50,0) *\cir<2pt>{} ="e",
\POS (60,0) *\cir<2pt>{} ="f",
\POS (70,5) *{\bullet} ="g",
\POS (70,-5) *{\bullet} ="h",

\POS "a" \ar@{-}^<<<{}_<<{}  "b",
\POS "b" \ar@{-}^<<<{}_<<{}  "c",
\POS "c" \ar@{.}^<<<{}_<<{}  "d",
\POS "d" \ar@{-}^<<<{}_<<{}  "e",
\POS "e" \ar@{-}^<<<{}_<<{} "f",
\POS "f" \ar@{-}^<<{}_<<<{} "g",
\POS "f" \ar@{-}^<<<{}_<<{} "h",
\endxy
\]
Set $\Theta = \Delta \setminus \alpha_{i} $, where $i=1, 2, \dots, n-2$ and $\alpha_{i}=e_{i}-e_{i+1}.$ Then $M'=M'_{\Theta}$ is of the form $GL_{i}(F) \times G'(n-i).$ For the case $i=n-2,$ $G'(n-i)=G'(2),$ and we have an exact sequence of groups
\[
1 \rightarrow \ZZ/2\ZZ  \rightarrow 
SL_1(D_2)  \times SL_1(D_2)  \rightarrow G'(2) \rightarrow 1.
\]
Further, for the case $i=n-3,$ we have $G'(n-i)=G'(3),$ with the following exact sequence
\[
1 \rightarrow \ZZ/2\ZZ  \rightarrow 
SL_2(D_2)  \rightarrow G'(3) \rightarrow 1.
\]

\item $\mathbf{n}$\textbf{: odd}

\[
\xy
\POS (10,0) *{\bullet} ="a" ,
\POS (20,0) *\cir<2pt>{} ="b" ,
\POS (30,0) *{\bullet} ="c" ,
\POS (40,0) *\cir<2pt>{} ="d" ,
\POS (50,0) *\cir<0pt>{} ="e",
\POS (60,0) *\cir<0pt>{} ="f",
\POS (70,0) *{\bullet} ="g",
\POS (80,0) *\cir<2pt>{} ="h",
\POS (90,0) *{\bullet} ="i",
\POS (100,5) *{\bullet} ="j",
\POS (100,-5) *{\bullet} ="k",

\POS "a" \ar@{-}^<<<{}_<<{}  "b",
\POS "b" \ar@{-}^<<<{}_<<{}  "c",
\POS "c" \ar@{-}^<<<{}_<<{}  "d",
\POS "d" \ar@{-}^<<<{}_<<{}  "e",
\POS "e" \ar@{.}^<<<{}_<<{}  "f",
\POS "f" \ar@{-}^<<<{}_<<{} "g",
\POS "g" \ar@{-}^<<<{}_<<{} "h",
\POS "h" \ar@{-}^<{}_<<<{} "i",
\POS "i" \ar@{-}^<<<{}_<<{} "j",
\POS "i" \ar@{-}^<<<{}_<<{} "k",
\endxy
\]
Set $\Theta = \Delta \setminus  \alpha_{i} $, where $i=2, 4, \dots, n-3$ and $\alpha_{i}=e_{i}-e_{i+1}.$ Then $M'=M'_{\Theta}$ is of the form $GL_{i/2}(D_2)  \times G'(n-i).$ For the case $i=n-3,$ $G'(n-i)=G'(3)=PSL_1(D_4).$ 

\item $\mathbf{n}$\textbf{: even}

\[
\xy
\POS (10,0) *{\bullet} ="a" ,
\POS (20,0) *\cir<2pt>{} ="b" ,
\POS (30,0) *{\bullet} ="c" ,
\POS (40,0) *\cir<2pt>{} ="d" ,
\POS (50,0) *\cir<0pt>{} ="e",
\POS (60,0) *\cir<0pt>{} ="f",
\POS (70,0) *{\bullet} ="g",
\POS (80,0) *\cir<2pt>{} ="h",
\POS (90,5) *{\bullet} ="i",
\POS (90,-5) *\cir<2pt>{} ="j",

\POS "a" \ar@{-}^<<<{}_<<{}  "b",
\POS "b" \ar@{-}^<<<{}_<<{}  "c",
\POS "c" \ar@{-}^<<<{}_<<{}  "d",
\POS "d" \ar@{-}^<<<{}_<<{}  "e",
\POS "e" \ar@{.}^<<<{}_<<{}  "f",
\POS "f" \ar@{-}^<{}_<<<{} "g",
\POS "g" \ar@{-}^<{}_<<<{} "h",
\POS "h" \ar@{-}^<<{}_<<<{} "i",
\POS "h" \ar@{-}^>{}_>>{} "j",
\endxy
\]
Set $\Theta = \Delta \setminus  \alpha_{i} $, where $i=2, 4, \dots, n-2, n$ and $\alpha_{i}= e_{i}-e_{i+1}$ for $i \neq n;$ $\alpha_{n}=e_{n-1} + e_{n}$ for $i=n.$ Then $M'=M'_{\Theta}$ is of the form $GL_{i/2}(D_2)  \times SU^{-}_{n-i}(D_2) .$ Here, $SU^{-}_{2k}(D_2) $ is defined as
\[
SU^{-}_{2k}(D_2)  := \{ g \in GL_{2k}(D_2)  :  g^{*} J^- g = J^- \},
\]
where $g^{*} = (g_{ij})^{*} = ( \bar{g}_{ji})$  with the usual involution $\bar{g}$ on $D$,  $J^- = \bigl(\begin{smallmatrix} 0 & I_k \\  -I_k & 0 \end{smallmatrix} \bigr),$ and $SU^{-}_{0}(D_2) =G(0)=1$ (see \cite[Section 2.3.3]{pr94}). We remark that, for even $n$, the $F$-inner form $G'=G'(n)$ of $SO_{2n}$ is of the form $SU^{-}_{n}(D_2) .$ For the case $i=n,$ $M'=M'_{\Theta}$ is of the form $GL_{n/2}(D_2) .$ Further, for the case $i=n-2,$ we have $M'_{\Theta} \simeq GL_{(n-2)/2}(D_2)  \times SU^{-}_{2}(D_2) $ and we have an exact sequence of groups
\[
1 \rightarrow \ZZ/2\ZZ \rightarrow SL_1(D_2)  \times SL_2 \rightarrow SU^{-}_{2}(D_2)  \rightarrow 1.
\]
\end{itemize}
\bigskip

\noindent \textbf{$\bold{{^2}D_n}$-type :} According to the parity of $n,$ there is only one (up to isomorphism) non-split $F$-inner form $\bG'=\bG'(n)$ of a quasi-split group $SO^*_{2n}$ over a quadratic extension $E/F$ with the following diagram.

\begin{itemize}
\item $\mathbf{n}$\textbf{: odd}

\[
\xy
\POS (10,0) *{\bullet} ="a" ,
\POS (20,0) *\cir<2pt>{} ="b" ,
\POS (30,0) *{\bullet} ="c" ,
\POS (40,0) *\cir<2pt>{} ="d" ,
\POS (50,0) *\cir<0pt>{} ="e",
\POS (60,0) *\cir<0pt>{} ="f",
\POS (70,0) *{\bullet} ="g",
\POS (80,0) *\cir<2pt>{} ="h",
\POS (90,0) *{\bullet} ="i",
\POS (100,5) *\cir<2pt>{} ="j",
\POS (100,-5) *\cir<2pt>{} ="k",

\POS "a" \ar@{-}^<<<{}_<<{}  "b",
\POS "b" \ar@{-}^<<<{}_<<{}  "c",
\POS "c" \ar@{-}^<<<{}_<<{}  "d",
\POS "d" \ar@{-}^<<<{}_<<{}  "e",
\POS "e" \ar@{.}^<<<{}_<<{}  "f",
\POS "f" \ar@{-}^<<<{}_<<{} "g",
\POS "g" \ar@{-}^<<<{}_<<{} "h",
\POS "h" \ar@{-}^<{}_<<<{} "i",
\POS "i" \ar@{-}^<<<{}_<<{} "j",
\POS "i" \ar@{-}^<<<{}_<<{} "k",
\POS "j" \ar@/{\;}^1pc/ "k",
\POS "k" \ar@/_1pc/ "j",
\endxy
\]
Set $\Theta = \Delta \setminus \alpha_{i} $, where $i=2, 4, \dots, n-1, n$ and $\alpha_{i}=e_{i}-e_{i+1}$ for $i \neq n;$ $\alpha_{n}=e_{n-1}+e_{n}$ for $i=n.$ When $i=2, 4, \dots, n-3,$ $M'=M'_{\Theta}$ is of the form $GL_{i/2}(D_2)  \times G'(n-i).$  When $i=n-1,$ $M'=M'_{\Theta}$ is of the form $GL_{i/2}(D_2)  \times E^{\times}.$ When $i=n,$ $M'=M'_{\Theta}$ is of the form $GL_{(n-1)/2}(D_2) .$

\item $\mathbf{n}$\textbf{: even}

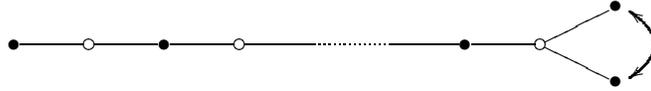
\[
\xy
\POS (10,0) *{\bullet} ="a" ,
\POS (20,0) *\cir<2pt>{} ="b" ,
\POS (30,0) *{\bullet} ="c" ,
\POS (40,0) *\cir<2pt>{} ="d" ,
\POS (50,0) *\cir<0pt>{} ="e",
\POS (60,0) *\cir<0pt>{} ="f",
\POS (70,0) *{\bullet} ="g",
\POS (80,0) *\cir<2pt>{} ="h",
\POS (90,5) *{\bullet} ="i",
\POS (90,-5) *{\bullet} ="j",

\POS "a" \ar@{-}^<<<{}_<<{}  "b",
\POS "b" \ar@{-}^<<<{}_<<{}  "c",
\POS "c" \ar@{-}^<<<{}_<<{}  "d",
\POS "d" \ar@{-}^<<<{}_<<{}  "e",
\POS "e" \ar@{.}^<<<{}_<<{}  "f",
\POS "f" \ar@{-}^<{}_<<<{} "g",
\POS "g" \ar@{-}^<{}_<<<{} "h",
\POS "h" \ar@{-}^<<{}_<<<{} "i",
\POS "h" \ar@{-}^>{}_>>{} "j",

\POS "i" \ar@/{\;}^1pc/ "j",
\POS "j" \ar@/_1pc/ "i",

\endxy
\]
Set $\Theta = \Delta \setminus  \alpha_{i} $, where $i=2, 4, \dots, n-2$ and $\alpha_{i}= e_{i}-e_{i+1}.$  Then $M'=M'_{\Theta}$ is of the form $GL_{i/2}(D_2)  \times G'(n-i).$
\end{itemize}
\subsection{$R$-groups} \label{section for def of R}
In this section, we recall the definitions of Knapp-Stein, Arthur and endoscopic $R$--groups.
For $\sigma \in \Irr(M)$ and $w\in W_M,$ we let ${^w}\sigma$ be the representation given by ${^w}\sigma(x)=\sigma(w^{-1}xw).$ (Note, for the purpose of computing $R$--groups, we need not worry about the representative we choose for $w.$)
Given $\sigma \in \Pi_{\disc}(M),$ we define
\[ W(\sigma) := \{ w \in W_M : {^w}\sigma \s \sigma \}.
\]
Set $\Delta'_\sigma = \{ \alpha \in \Phi(P, A_M) : \mu_{\alpha} (\sigma) = 0 \},$ where $\mu_{\alpha} (\sigma)$ is the rank one Plancherel measure for $\sigma$ attached to $\alpha$ \cite[p.1108]{goldberg-class}. 
The \textit{Knapp-Stein $R$-group} is defined by
\[
R_{\sigma} := \{ w \in W(\sigma) : w \alpha > 0, \; \forall \alpha \in \Delta'_\sigma \}.
\]
Denote by $W'_{\sigma}$ the subgroup of $W(\sigma),$ generated by the reflections in the roots of $\Delta'_\sigma.$
We write $C(\sigma):={\End}_{G}(\ii_{G,M} (\sigma))$ for the algebra of $G$-endomorphisms of $\ii_{G,M} (\sigma),$ known as the commuting algebra of $\ii_{G,M} (\sigma).$

\begin{thm}[Knapp-Stein \cite{ks72}; Silberger \cite{sil78, sil78cor}]  \label{thm for Knapp-Stein-Sil}
For any $\sigma \in \Pi_{\disc}(M),$ we have
\[
W(\sigma) = R(\sigma) \ltimes W'_{\sigma}. 
\]
Moreover, $C(\sigma) \s \CC[R(\sigma)]_{\eta},$ the group algebra of $R(\sigma)$ twisted by a $2$-cocycle $\eta,$ which is explicitly defined in terms of group $W(\sigma).$
\end{thm}

Let $\phi : W_F \times SL_2(\CC) \rightarrow \widehat{M} \hookrightarrow \widehat{G}$ be an $L$-parameter.
We denote by $C_{\phi}(\widehat{G})$ the centralizer of the image of $\phi$ in $\widehat{G}$ and by $C_{\phi}(\widehat{G})^{\circ}$ its identity component. Fix a maximal torus $T_{\phi}$ in $C_{\phi}(\widehat{G})^{\circ}.$ 
\begin{rem} \label{remark for T}
From \cite[Chapter 2.4]{art12} we set a maximal torus $T_{\phi}$ in $C_{\phi}(\widehat{M})^{\circ}$ to be the identity component
\[
A_{\widehat M} = (Z(\widehat{M})^{\Gamma})^{\circ}
\]
of the $\Gamma$--invariants of the center $Z(\widehat{M}).$ 
\end{rem}
We set
\[
W_{\phi}^{\circ} := N_{C_{\phi}(\widehat{G})^{\circ}} (T_{\phi}) /  Z_{C_{\phi}(\widehat{G})^{\circ}} (T_{\phi}), \quad W_{\phi} := N_{C_{\phi}(\widehat{G})} (T_{\phi}) /  Z_{C_{\phi}(\widehat{G})} (T_{\phi}).
\]
The \textit{endoscopic $R$-group} $R_{\phi}$ is defined as follows
\[
R_{\phi}:=W_{\phi}/W_{\phi}^{\circ}.
\]
Note that $W_{\phi}$ can be identified with a subgroup of $W_M$ (see \cite[p.45]{art89ast}). Let $\Pi_{\phi}(M) \subset \Irr(M)$ be the $L$-packet associated to the $L$-parameter $\phi.$ For $\sigma \in \Pi_{\phi}(M),$ we set
\begin{equation} \label{def of W_phi, sigma}
W_{\phi, \sigma}^{\circ} :=  W_{\phi}^{\circ} \cap W(\sigma),~  W_{\phi, \sigma} :=  W_{\phi} \cap W(\sigma) ~ \text{and}~ R_{\phi, \sigma}:=W_{\phi, \sigma}/W_{\phi, \sigma}^{\circ}.
\end{equation}
We call $R_{\phi, \sigma}$ the \textit{Arthur $R$-group}.
\section{Transfer of $R$--groups} \label{transfer of R}
We continue with the notation in Section \ref{pre}. Let $\bG=\bG(n)$ denote a quasi-split classical group, $SO_{2n+1},$ $Sp_{2n},$ $SO_{2n},$ or $SO_{2n}^*,$ of rank $n$ over $F,$ and let $\bG'=\bG'(n)$ denote an $F$--inner form of $\bG$ ($\bG'$ can be $\bG$ itself).
In this section, we describe Weyl group actions for $\bG'$. 
Further, we prove the three $R$--groups for $\bG'$, Knapp-Stein, Arthur, and endoscopic, are identical if they are associated to the same elliptic tempered $A$--parameter for an $F$-Levi subgroup of $\bG'.$ 
For these cases, all three $R$--groups are thus invariant on $A$--packets and preserved by inner forms. 
As a corollary, we identically transfer all the description of Knapp-Stein $R$--groups for the quasi-split $\bG$ (see \cite{herb93, goldberg-class} and Appendix \ref{R-gps for SO*}) to its inner form $\bG'.$

Following Arthur's endoscopic classification for $G$ and $G'$ in \cite{art12}, we take an $L$-group
\[
{^L}{G}={^L}{G'}=\widehat{G} \rtimes \Gal(E/F),
\]
where $E=F$ when $\bG$ is split, or $E$ a quadratic extension $E$ over $F$ when $\bG=SO^*_{2n}.$ To be precise, we have following $L$--groups.
\begin{enumerate}
\item Type $\bold{B_n}:$ $\widehat{G}=\widehat{G'}=Sp_{2n}(\CC)={^L}{G}={^L}{G'}.$
\item Type $\bold{C_n}:$ $\widehat{G}=\widehat{G'}=SO_{2n+1}(\CC)={^L}{G}={^L}{G'}.$
\item Type $\bold{{^1}D_n}:$ $\widehat{G}=\widehat{G'}=SO_{2n}(\CC)={^L}{G}={^L}{G'}.$ 
\item Type $\bold{{^2}D_n}:$ $\widehat{G}=\widehat{G'}=SO_{2n}(\CC),$ and ${^L}{G}={^L}{G'} = SO_{2n}(\CC) \rtimes \Gal(E/F).$ 
Note that $\Gal(E/F) \s O_{2n}(\CC)/SO_{2n}(\CC),$ 
where $O_{2n}(\CC)$ denotes the even orthogonal group of size $2n.$ 
Further, the nontrivial outer automorphism exchanging the roots $\alpha_n$ and $\alpha_{n-1}$ is induced by conjugation by an element in $O_{2n}(\CC) \setminus  SO_{2n}(\CC),$ cf. \cite[Section 7]{cpssh11}.
\end{enumerate}
Let $\bM$ and $\bM'$ be $F$--Levi subgroups of $\bG$ and $\bG',$ respectively, such that $\bM'$ is an $F$--inner form of $\bM.$
We identify
\begin{equation*}  \label{iso between L-groups}
\widehat M = \widehat{M'}, \quad {^L}{M}={^L}{M'}.
\end{equation*}
\subsection{Elliptic tempered $A$-packets} \label{elliptic tempered a-packets}
In this section, we describe elliptic tempered $A$-packets for $M$ and $M'.$ We follow Arthur's local results in \cite[Chapters 1 and 9]{art12}, cf. \cite{js04, mw03}.
Let $\phi : W_F \times SL_2(\CC) \rightarrow {^L}{M}$ be an elliptic tempered $L$-parameter \cite[Section 8.2]{bo79}. 
We recall from \cite{art89ast, art12} that $\phi$ is \textit{elliptic} if the quotient group $C_{\phi}(\widehat{M}) / Z({\widehat{M}})$ is finite, 
and $\phi$ is \textit{tempered} if the image $\phi(W_F)$ is relatively compact (or bounded).
Considering \eqref{form of M} and \eqref{form of M'}, 
we note that $\phi$ is of the form $\phi_1 \oplus \phi_2 \oplus \cdots \oplus \phi_k \oplus \phi_-,$ 
where $\phi_i$ is an elliptic tempered $L$-parameter for $GL_{n_i}(F)$ as well as $GL_{m_i}(D),$ 
and $\phi_-$ is that for $G_-(m)$ as well as $G'_-(m).$
Due to \cite{he00}, \cite{ht01}, \cite{hs11}, and \cite[Theorems 1.5.1 and 9.4.1]{art12}, we construct $A$-packets $\Pi_{\phi}(M)$ and $\Pi_{\phi}(M')$ of $M$ and $M'$ associated to $\phi,$ which respectively consist of discrete series representations of $M$ and $M'.$ 
\begin{rem} \label{rem for A-packets}
We make the following remarks on the elliptic and tempered $A$-packets.
\begin{enumerate}
\item[1.] For our purpose, we deal with discrete series representations for $M$ and $M'.$ So, we do not need the extra $SU(2)$ factor for parameters, in which case $A$-parameters coincide with $L$-parameter except for  the $O_{2n}(\CC)$-- conjugation for even orthogonal cases (see 3. below). Arthur called the $A$-parameter which is trivial on the extra $SU(2)$ factor to be generic \cite[Chapter 1.4]{art12}.

\item[2.] For the case $\bG=SO_{2n+1},$ $Sp_{2n},$ and their inner forms $\bG',$ the elliptic and tempered $A$-packets consist of irreducible discrete series representations, 
so they coincide with elliptic and tempered $L$--packets. 

\item[3.] For the case $\bG=SO_{2n},$ $SO^*_{2n}$ those $A$-packets consist of $O_{2n}(F) /  SO_{2n}(F)$--orbits of order 2 or 1 
(hence, not necessarily individual representations) of irreducible discrete series representations under the $O_{2n}(F)$--action
by conjugation on $G.$ The same is true for $F$-inner forms $\bG'$ of $SO_{2n}$ and $SO^*_{2n},$ cf. \cite[Chapters 1.5 and 9.4]{art12}.

\item[4.] Arthur's classifications for $SO_{2n}$ and $SO^*_{2n}$ rely on some expected arguments on the stabilization of the twisted trace formula. So, as Arthur did in \cite{art12}, we shall assume this in Hypotheses \ref{hyp}.a below.

\item[5.] The classification for non quasi-split $F$-inner forms $\bG'$ of $\bG$ is simply stated in \cite[Theorem 9.4.1]{art12} and the proof is expected to be delivered in future work of Arthur or Kaletha, et al.. So, we will assume this in Hypotheses \ref{hyp}.b below.
\end{enumerate} 
\end{rem}
\begin{hyp} \label{hyp}
\begin{enumerate}[(a)]
\item We assume that the twisted trace formula for $GL_n$ and a twisted even orthogonal group (whose identity component is $SO_{2n}$ or $SO^*_{2n}$) can be stabilized \cite[Chapter 3.2]{art12}.

\item We assume that the classification for non quasi-split $F$-inner forms $\bG'$ of $\bG$ is known \cite[Chapter 9.4]{art12}.
\end{enumerate}
\end{hyp}
Following \cite[Chapter 9.2]{art12}, we review the component groups parameterizing $A$--packets and their connection with endoscopic $R$--groups. 
Let $S_{\phi}(\widehat{M})$ denote the group $\pi_0(C_{\phi}(\widehat{M})) := C_{\phi}(\widehat{M}) / C_{\phi}(\widehat{M})^{\circ}$ of connected components. 
Since $\phi$ is elliptic, the identity component $S_{\phi}(\widehat M)^{\circ}$ is contained in $Z(\widehat{M})^{\Gamma}$ (see \cite[Lemma 10.3.1]{kot84} and \cite[Section 7]{art89ast}).
Thus, we have
\[
S_{\phi}(\widehat{M}) =  C_{\phi}(\widehat M) / Z(\widehat{M})^{\Gamma}, 
\]
which can be considered as a finite subgroup of $(\widehat{M})_{\ad} := \widehat{M} / Z(\widehat{M}).$ It turns out that $S_{\phi}(\widehat{M}) $ is a finite abelian $2$-group for our cases \cite[Chapter 1.4]{art12}.
Let $S_{\phi, \scn}(\widehat{M})$ be the full pre-image of $C_{\phi}(\widehat M) / Z(\widehat{M})^{\Gamma}$ in the simply connected cover $\widehat{M}_{\scn}$ of the derived group $\widehat{M}_{\der}$ of $\widehat{M}$ under the isogeny $\widehat{M}_{\scn} \twoheadrightarrow (\widehat{M})_{\ad}$. One thus has the following exact sequence:
\begin{equation} \label{exact for M'}
1 \rightarrow Z(\widehat M_{\scn}) \rightarrow S_{\phi, \scn}(\widehat{M}) \rightarrow S_{\phi}(\widehat{M}) \rightarrow 1.
\end{equation}
It is well-known \cite[p.280]{kot97} that the $F$-inner form $\bM'$ of $\bM$ determines 
a unique character $\zeta_{\bM'}$ on $Z(\widehat M_{\scn})$ whose restriction to $Z(\widehat M_{\scn})^{\Gamma}$ corresponds to 
the isomorphism class of $\bM'$ via the Kottwitz isomorphism \cite[Theorem 1.2]{kot86}. 
We denote by $\Irr(S_{\phi, \scn}(\widehat{M}), \zeta_{\bM'})$ as the set of irreducible representations of $S_{\phi, \scn}(\widehat{M})$ 
whose restriction to $Z(\widehat M_{\scn})$ is equal to the character $\zeta_{\bM'}.$ 
We note that, for the case $\bM'=\bM,$  the character $\zeta_{\bM}$ turns out to be the trivial character $\mathbbm{1}.$ 
Given an elliptic tempered $L$-parameter $\phi$ for $M',$ there is a one-one bijection between $\Pi_{\phi}(M')$ and $\Irr(S_{\phi, \scn}(\widehat{M}), \zeta_{\bM'})$ (see \cite{he00}, \cite{ht01}, \cite{hs11}, and \cite[Theorems 1.5.1 and 9.4.1]{art12}). The following lemma is due to this bijection.
\begin{lm}  \label{lemma for singletons}
Let $\bM$ and $\bM'$ be as above. Given an elliptic tempered $L$-parameter $\phi$ for $M',$ if the $A$--packet $\Pi_{\phi}(M)$ is a singleton, then so is the $A$--packet $\Pi_{\phi}(M').$
\end{lm}
\begin{proof}
It suffices to show that $\Irr(S_{\phi, \scn}(\widehat{M}), \zeta_{\bM'})$ is a singleton.
Note that, if $\bM=\bM',$ then $\Irr(S_{\phi, \scn}(\widehat{M}), \zeta_{\bM}) = \Irr(S_{\phi}(\widehat{M})).$
Since $\Pi_{\phi}(M)$ is a singleton, we have $\# \Irr(S_{\phi}(\widehat{M})) =1,$ which implies  
$S_{\phi}(\widehat{M})= \{ 1 \}.$ 
From the exact sequence \eqref{exact for M'}, 
we have $S_{\phi, \scn}(\widehat{M}) \s Z(\widehat M_{\scn}),$ which is finite and abelian. Thus, we have $\Irr(S_{\phi, \scn}(\widehat{M}), \zeta_{\bM'}) = \{ \zeta_{\bM'}   \}.$ This completes the proof.
\end{proof}
Through a natural embedding $\widehat{M} \hookrightarrow \widehat{G},$ we have
\begin{equation} \label{embedding to G}
\phi : W_F \times SL_2(\CC) \rightarrow {^L}{M} \hookrightarrow \,^L{G}.
\end{equation}
It follows that the $A$--packet $\Pi_{\phi}(G)$ of $G$ associated to $\phi$ is comprised of irreducible (hence, tempered) constituents of induced representations from all members in $\Pi_{\phi}(M)$ (see \cite[Theorem 1.5.1]{art12}). 
The same is true for the $A$--packet $\Pi_{\phi}(G')$ of $G'$ (see \cite[Theorem 9.4.1]{art12}). 
As mentioned in Remark \ref{rem for A-packets}.3, each member in $\Pi_{\phi}(G)$ and $\Pi_{\phi}(G')$ for even orthogonal groups represents the orbit of order 2 or 1 of irreducible tempered representations rather than an individual representation.
We consider $C_{\phi}(\widehat G) / Z(\widehat{G})^{\Gamma}$ as a subgroup of $(\widehat{G})_{\ad}.$ Write $C_{\phi, \scn}(\widehat{G})$ for the full pre-image of $C_{\phi}(\widehat G) / Z(\widehat{G})^{\Gamma}$ in $\widehat{G}_{\scn}$ under the isogeny $\widehat{G}_{\scn} \twoheadrightarrow (\widehat{G})_{\ad}.$ We then have the following exact sequence
\[
1 \rightarrow Z(\widehat G_{\scn}) \rightarrow C_{\phi, \scn}(\widehat{G}) \rightarrow C_{\phi}(\widehat G) / Z(\widehat{G})^{\Gamma} \rightarrow 1.
\]
Since $\phi$ may not be elliptic for $G,$ there is no guarantee that 
$C_{\phi}(\widehat G) / Z(\widehat{G})^{\Gamma}$ is finite. We let 
\begin{align*}
S_{\phi}(\widehat{G}) & :=  \pi_0(C_{\phi}(\widehat G) / Z(\widehat{G})^{\Gamma}) \\
S_{\phi, \scn}(\widehat{G}) & := \pi_0(C_{\phi, \scn}(\widehat{G})) \\
\widehat Z_{\phi, \scn} & := Z(\widehat G_{\scn}) / (Z(\widehat G_{\scn}) \cap C_{\phi, \scn}(\widehat{G})^{\circ}).
\end{align*}
We then have a central extension 
\[
1 \rightarrow \widehat Z_{\phi, \scn}  \rightarrow S_{\phi, \scn}(\widehat{G}) \rightarrow S_{\phi}(\widehat{G}) \rightarrow 1,
\]
cf. \cite[(9.2.2)]{art12}.
Let $\zeta_{\bG'}$ be a unique character on $Z(\widehat G_{\scn})$ whose restriction to $Z(\widehat G_{\scn})^{\Gamma}$ corresponds to 
the class of the $F$-inner form $\bG'$ of $\bG$ via the Kottwitz isomorphism \cite[Theorem 1.2]{kot86}. 
We denote by $\Irr(S_{\phi, \scn}(\widehat{G}), \zeta_{\bG'})$ as the set of irreducible representations of $S_{\phi, \scn}(\widehat{G})$ 
with central character $\zeta_{\bG'}$ on  $Z(\widehat G_{\scn}).$ 
Given an elliptic tempered $L$-parameter $\phi$ for $M'$ which is tempered for $G'$ via the composite \eqref{embedding to G}, 
there is a one-one bijection between $\Pi_{\phi}(G')$ and $\Irr(S_{\phi, \scn}(\widehat{G}), \zeta_{\bG'})$ \cite[9.4.1]{art12}. 
Note that Lemma \ref{lemma for singletons} is also true for tempered $L$-parameters for $G$ and $G'.$
Further, the diagrams in \cite[(2.4.3) and (9.2.16)]{art12} yield
\begin{equation} \label{useful exact}
1 \longrightarrow S_{\phi}(\widehat{M})
\longrightarrow S_{\phi}(\widehat{G})
\longrightarrow R_{\phi}
\longrightarrow 1
\end{equation}
and
\begin{equation} \label{useful exact 2}
1 \longrightarrow S_{\phi, \scn}(\widehat{M})
\longrightarrow S_{\phi, \scn}(\widehat{G})
\longrightarrow R_{\phi}
\longrightarrow 1.
\end{equation}
\subsection{Weyl group actions} \label{weyl group actions}
In this section, we describe the action of the Weyl group $W_{M}$ (respectively, $ W_{M'},$ $W_{\widehat M}$) on $M$ (respectively, $M',$ $\widehat M$) and $\sigma$ (respectively, $\sigma',$ $\phi$). 
We identify $\widehat G = \widehat{G'}$ and $\widehat M = \widehat{M'}.$ 
Recall that
the Weyl groups are 
$W_M = W(\bG, \bA_\bM) := N_\bG(\bA_\bM) / Z_\bG(\bA_\bM),$
$W_{M'} = W(\bG', \bA_{\bM'}) := N_{\bG'}(\bA_{\bM'}) / Z_{\bG'}(\bA_{\bM'}),$ and
$W_{\widehat M} = W(\widehat G, A_{\widehat M}) := N_{\widehat G}(A_{\widehat M}) / Z_{\widehat G}(A_{\widehat M}).$
Through the duality
\begin{equation} \label{duality}
s_{\alpha} \mapsto s_{\alpha^{\vee}}
\end{equation}
between simple reflections for $\alpha \in \Delta,$ we have 
\begin{equation}  \label{two isos}
W_M \s W_{\widehat M} \quad W_{\widehat{M}'} \s W_{M'},
\end{equation}
cf. \cite[Chapter 2.4]{art12}.  
We thus identify
\begin{equation}  \label{identity for Weyls}
W_M = W_{\widehat M}=W_{\widehat{M}'} =W_{M'}.
\end{equation}
We start with the following lemma which will be used to describe the actions of Weyl groups.
\begin{lm} \label{conjugation lemma}
Denote by $D$ a central division algebra (possibly, $F$ itself) of dimension $d^2$ over $F.$ Let $\pi \in \Irr(GL_m(D))$ be given. Then we have
\begin{equation} \label{iso among several reps}
\widetilde{\pi} \s {^t}{\pi}^{-1} \s \bar{\pi}^{-1} \s  {^t}\bar{\pi}^{-1},
\end{equation}
where ${^t}{\pi}^{-1}(g):=\pi({^t}g^{-1}),$ $\bar{\pi}^{-1}(g):=\pi(\bar{g}^{-1}),$ and $\bar{g}$ denotes the usual involution of $g$ in $GL_m(D),$ unless $D=F,$ in which $\bar{g}=g.$
\end{lm}
\begin{proof}
Following the idea in the proof of \cite[Lemma 2.2]{ms00}, it suffices to prove that all representations in \eqref{iso among several reps} have the same (Harish-Chandra) character over regular semi-simple elements.
Given a regular semi-simple element $g \in GL_m(D),$ we note that all ${^t}g, \bar g,$ and ${^t}\bar{g}$ are conjugate over $\bar F$ 
because they are regular semisimple and have the same characteristic polynomial over $F.$
Since the character $\Theta_{\widetilde{\pi}}(g)$ equals $\Theta_{\pi}(g^{-1})$ \cite[Introduction d.]{dkv} and Harish-Chandra characters are invariant under conjugation by $F$-points,
it thus remains to show that ${^t}g, \bar g$ and ${^t}\bar{g}$ are all conjugate over $F.$ 

Let $x$ and $y$ be two $\bar F $--conjugate regular semisimple elements in $GL_m(D).$ 
We denote by $\bG'$ the $F$--inner form of $GL_n$ with $\bG'(F)=GL_m(D).$
Since $\bG'(\bar F)=GL_n(\bar F)$ with $n=md,$
we then have $h \in GL_n(\bar F)$ such that $x = hyh^{-1}.$  
It then follows that
\[
h y h^{-1}=x={^\upsilon x}=({^\upsilon h})y({^\upsilon h^{-1}})
\]
for any $\upsilon \in \Gamma,$ so that 
$h^{-1}({^\upsilon h})$ lies in the centralizer $Z_{GL_n(\bar F)}(y)$ of $y$ in $GL_n(\bar F).$ 
Thus, the mapping $\upsilon \mapsto h^{-1}({^\upsilon h})$ from $\Gamma$ to $Z_{GL_n(\bar F)}(y)$ gives a $1$--cocycle class in the Galois cohomology
$H^1(F, Z_{GL_n(\bar F)}(y)).$
We further note from \cite[Section 3]{kot82} that 
\[
\ker \Big( H^1(F, Z_{GL_n(\bar F)}(y)) \rightarrow H^1(F, \bG') \Big)
\]
is in the bijection with the set of $F$--conjugacy classes in $\bar F$--conjugacy classes of $y$. 
Since $y$ is a regular semisimple element in $GL_m(D),$ 
it follows that $Z_{GL_n(\bar F)}(y)$ forms a torus which is isomorphic to $\prod_{j} E_j^{\times}$ over $F,$ where $E_j$ is an extension of $F$ of degree $[E_j:F]$ with $\sum [E_j:F] = m.$ 
By Hilbert's Theorem $90,$ we have $H^1(F, Z_{GL_n(\bar F)}(y))=1,$ which implies that $x$ and $y$ are conjugate over $F.$
We thus conclude that all ${^t}g, \bar g,$ and ${^t}\bar{g}$  are conjugate over $F.$
\end{proof}
\subsubsection{Weyl group actions on Levi subgroups and their representations} \label{WM-actions on Levis}
For simplicity, in Section \ref{WM-actions on Levis}, we will write $\bG$ for both quasi-split classical groups $SO_{2n+1},$ $Sp_{2n},$ $SO_{2n},$ $SO_{2n}^*$ and their non quasi-split $F$--inner forms.
We describe the action of $W_M$ on an irreducible representation $\sigma \in \Irr(M),$ based on the results in \cite{goldberg-class} for split cases.
Recall from Section \ref{structure} that any of $M$ is of the form
\[
GL_{n_1}(D) \times GL_{n_2}(D) \times \cdots \times GL_{n_k}(D) \times G_{-}(m),
\]
where $D$ is a central division algebra of dimension $1$ or $4$ over $F$ and $\bG_{-}$ is the same type as $\bG$ with lower rank $m.$ 
We denote by $S_k$ the symmetric group in $k$ letters.
Since we have from \cite[Section 2]{goldberg-class} and Appendix \ref{R-gps for SO*}
\begin{equation} \label{desc WM}
W_M \subset S_k \ltimes \ZZ_2^k,
\end{equation}
the identification \eqref{identity for Weyls} implies that \eqref{desc WM} is true for non quasi-split cases.
More precisely, $W_M\simeq S\ltimes\sC,$ 
where  $S=\langle (ij)|n_i=n_j \rangle,$ and $\sC \subset \ZZ_2^k.$ 
For $g \in M,$ write
\[
g=(g_1,\dots g_i,\dots,g_j,\dots, g_k,h) \in GL_{n_1}(D) \times GL_{n_2}(D) \times \cdots \times GL_{n_k}(D) \times G_{-}(m).
\]
The permutation $(ij)$ acts on $g \in M$
\begin{equation}  \label{action ij}
(ij):(g_1,\dots g_i,\dots,g_j,\dots, g_k,h) \mapsto (g_1,\dots,g_j,\dots,g_i,\dots,g_k,h).
\end{equation}
The finite $2$-group $\ZZ_2^k$ is generated by elements $C_i$ which acts on $g \in M$
\begin{equation}  \label{action Ci}
C_i:(g_1,\dots,g_i,\dots,g_k,h) \mapsto (g_1,\dots,\,{^t}\bar g_i^{-1},\dots,g_k,h),
\end{equation}
where $~~\bar{~~}~~$ denotes the involution  in Lemma \ref{conjugation lemma}, which is trivial for quasi-split cases.
Moreover, 
if $\bG$ is of type $\bold{B_n}$ or $\bold{C_n},$ then $\sC=\ZZ_2^k,$ cf. \cite[Section 2]{herb93}.
If $\bG$ is of type $\bold{D_n}$ (i.e., either $\bold{^1 D_n}$ or $\bold{^2 D_n}$), then $\sC=\sC_1\times\sC_2,$ where $\sC_1=\langle C_i|n_i\text{ is even}\rangle,$ and
\begin{equation} \label{C_2}
\sC_2=\begin{cases} \langle C_iC_j|n_i,n_j\text{ are odd}\rangle, &\text{if }m=0;\\\langle C_ic_0|n_i\text{ is odd}\rangle, &\text{if }m>0,
\end{cases}
\end{equation}
where $c_0$ is the outer automorphism given by the reflection $\alpha_{n-1} \leftrightarrow \alpha_n$ of the Dynkin diagram, cf. \cite[Section 3]{herb93}. 
Set $\sigma$ to be $\sigma_1 \otimes \sigma_2 \otimes \cdots \otimes \sigma_k \otimes \tau.$
From Lemma \ref{conjugation lemma}, \eqref{action ij}, and \eqref{action Ci}, we then have
\begin{gather*}
(ij)\sigma=\sigma_1\otimes\sigma_j\otimes\cdots
\otimes\sigma_i\otimes\cdots
\otimes\sigma_k\otimes\tau;\\
C_i\sigma=\sigma_1\otimes\cdots
\otimes\tilde\sigma_i\otimes\cdots
\otimes\sigma_k\otimes\tau;\\
C_ic_0\sigma=\sigma_1\otimes\cdots
\otimes\tilde\sigma_i\otimes\cdots
\otimes\sigma_k\otimes c_0\tau,
\end{gather*}
and these describe the action of $W_M$ on $\sigma.$
\subsubsection{Weyl group actions on $L$-group $\widehat M$ and $L$--parameter $\phi$} \label{action on parameters}
In this section we describe the action of the Weyl group $W_{\widehat{M}}$
on $L$--group $\widehat M$
and on the elliptic tempered parameter $\phi.$
From \eqref{identity for Weyls} and \eqref{desc WM}, we have
\[
W_{\widehat M} \subset S_k \ltimes \ZZ_2^k.
\]
More precisely, $W_{\widehat M}\simeq S\ltimes\sC,$ 
where  $S=\langle (ij)|n_i=n_j \rangle,$ and $\sC \subset \ZZ_2^k.$ 
For $\hat g \in \widehat M,$ write
\[
\hat g = (\hat g_1,\dots \hat g_i,\dots,\hat g_j,\dots, \hat g_k,\hat h) \in GL_{n_1}(\CC) \times GL_{n_2}(\CC) \times \cdots \times GL_{n_k}(\CC) \times \widehat G_{-}(m).
\]
By the duality \eqref{duality}, 
The permutation $(ij)$ acts on $\hat g \in \widehat M$
\begin{equation}  \label{action ij L-group}
(ij):(\hat g_1,\dots \hat g_i,\dots,\hat g_j,\dots, \hat g_k,\hat h) \mapsto (\hat g_1,\dots,\hat g_j,\dots,\hat g_i,\dots,\hat g_k,\hat h).
\end{equation}
The finite $2$-group $\ZZ_2^k$ is generated by elements $C_i$ which acts on $\hat g \in \widehat M$
\[
C_i:(\hat g_1,\dots,\hat g_i,\dots,\hat g_k,\hat h)\mapsto (\hat g_1,\dots,\,{^t}{\hat g}_i^{-1},\dots,\hat g_k,\hat h).
\]
Moreover, 
if $\bG$ is of type $\bold{B_n}$ or $\bold{C_n},$ then $\sC=\ZZ_2^k.$
If $\bG$ is of type $\bold{D_n}$ (i.e., either $\bold{^1 D_n}$ or $\bold{^2 D_n}$), then $\sC=\sC_1\times\sC_2,$ where $\sC_1=\langle C_i|n_i\text{ is even}\rangle,$ and
\begin{equation}  \label{C_2 for L}
\sC_2=\begin{cases} \langle C_iC_j|n_i,n_j\text{ are odd}\rangle, &\text{if }m=0;\\\langle C_ic_0|n_i\text{ is odd}\rangle, &\text{if }m>0,
\end{cases}
\end{equation}
where $c_0$ is the outer automorphism given by the reflection $\alpha^{\vee}_{n-1}\leftrightarrow \alpha^{\vee}_n$ of the Dynkin diagram. Set $\phi$ to be $\phi_1 \oplus \phi_2 \oplus \cdots \oplus \phi_k \oplus \phi_-.$
Since $\phi_i({^t}{\hat g}_i^{-1}) = \tilde{\phi}({\hat g}_i)$ for $\hat g_i \in GL_{n_i}(\CC)$ \cite{ht01, he00, adams-vogan}, where $\tilde{\phi}$ is the contragredient of $\phi,$ we have
\begin{gather*}
(ij)\phi=\phi_1\oplus\phi_j\oplus\cdots
\oplus\phi_i\oplus\cdots\oplus\phi_k
\oplus\phi_{-};\\
C_i\phi=\phi_1\oplus\cdots
\oplus\tilde\phi_i\oplus\cdots
\oplus\phi_k\oplus\phi_{-};\\
C_ic_0  \phi=\phi_1\oplus\cdots
\oplus\tilde\phi_i\oplus\cdots
\oplus\phi_k\oplus c_0\phi_{-},
\end{gather*}
and these describe the action of $W_{\widehat M}$ on $\phi.$
\begin{rem}  \label{remark on c_0}
For the set $\sC_2$ defined in \eqref{C_2} and \eqref{C_2 for L}, 
we have have either $m>0$ or each $m_i$ is even due to Section \ref{structure}. 
It hence follows that $\{C_iC_j|n_i,n_j\text{ are odd}\}$ is empty, 
and that $\{C_ic_0|n_i\text{ is odd}\}$ is non-empty if only if $\bG$ is of type labelled ``n: any'' of $\bold{^1 D_n}$ in Section \ref{structure}.
\end{rem}
\subsection{Main theorem} \label{main theorem}
In this section, we prove the following theorem, which asserts that the three $R$--groups, Knapp-Stein, Arthur and endoscopic, are identical if they are attached to the same elliptic tempered $L$--parameter.
\begin{thm} \label{iso bw R}
Let $\phi : W_F \times SL_2(\CC) \rightarrow {^L}{M}$ be an elliptic tempered $L$-parameter. Given any $\sigma \in \Pi_{\phi}(M)$ and $\sigma' \in \Pi_{\phi}(M'),$ we have
\[
R_{\sigma} \s R_{\phi, \sigma} \s R_{\phi} \s R_{\phi, \sigma'} \s R_{\sigma'}.
\]
\end{thm}
\begin{rem} \label{q-split and L-packet for SO}
\begin{itemize}
\item[1.] For the quasi-split case, the isomorphism
\[
R_{\sigma} \s R_{\phi, \sigma} \s R_{\phi}
\]
is deduced from the argument in \cite[Chapters 6.5 \& 6.6]{art12}. However, we will develop our own proof which works for both quasi-split and non quasi-split cases, 
except the argument in \cite[p.346]{art12} by Arthur that 
\begin{equation} \label{useful arthur's arg}
W'(\sigma) \supset W^{\circ}_{\phi}.
\end{equation}
This will be used to verify 
the containment \eqref{containment 2} below.

\item[2.] For $SO_{2n}$ and $SO^*_{2n},$ since our $L$-parameter $\phi$ for $M$ is a generic $A$-parameter as mentioned in Remark \ref{rem for A-packets}.1, it is true from \cite[Theorem 2.2.4]{art12} that 
\begin{equation*}  \label{c0 actions for M}
c_0 \phi \s \phi \Longrightarrow \# \Pi_{\phi}(M)=1. 
\end{equation*}
Moreover, by Lemma \ref{lemma for singletons}, we have
\begin{equation*}  \label{c0 actions for M'}
c_0 \phi \s \phi \Longrightarrow \# \Pi_{\phi}(M')=1.
\end{equation*}
Set $\Pi_{\phi}(M)=\{ \pi \}$ and $\Pi_{\phi}(M')=\{ \pi' \}.$
Thus, we have
\[
c_0 \phi \s \phi \Longrightarrow c_0 \pi \s \pi 
~~~ \text{and} ~~~
c_0 \pi' \s \pi'.
\]
\end{itemize}
\end{rem}
Combining Theorem \ref{iso bw R} with the second named author's result for $R_{\sigma}$ in \cite{goldberg-class}, 
we obtain the following corollaries which give a description of $R_{\sigma'}.$
\begin{cor}
Let $\phi : W_F \times SL_2(\CC) \rightarrow {^L}{M}$ be an elliptic tempered $L$-parameter. Let $\sigma \in \Pi_{\phi}(M)$ and $\sigma' \in \Pi_{\phi}(M')$ be given. Then, $\ii_{G,M}(\sigma)$ is irreducible if and only if $\ii_{G',M'}(\sigma')$ is irreducible.
\end{cor}
\begin{cor} \label{r-groups for G'=bc}
Let $\bM'$ be an $F$-Levi subgroup of $\bG'$ of type $\bold{B_n}$ or $\bold{C_n}.$ Let $\phi : W_F \times SL_2(\CC) \rightarrow {^L}{M}$ be an elliptic tempered $L$-parameter. Let $\sigma' \in \Pi_{\phi}(M')$ be given. Set 
\[
I(\sigma') = \{ 1 \leq i \leq k : \sigma'_i \s  \widetilde{\sigma}'_i ~~\text{and}~~ \ii_{G'(n_i + m), ~ GL(n_i)\times G'(m)} (\sigma' \otimes \tau') ~~\text{is reducible}~  \},
\]
and $d$ is the number of inequivalent $\sigma'_i$ such that $i \in I(\sigma').$ Then, we have
\[
R_{\sigma'} \s \ZZ_2^d,
\]
and $R_{\sigma'}$ is generated by  the $d$  sign changes 
\[
\{C_i | i \in I_{\sigma'},\text{ and }\sigma_j'\not\simeq\sigma_i'\text{ for all } j>i\}.
\]
\end{cor}
\begin{cor} \label{r-groups for G'=d}
Let $\bM'$ be an $F$-Levi subgroup of $\bG'$ of type $\bold{D_n}$ (i.e., either $\bold{^1 D_n}$ or $\bold{^2 D_n}$) .
Let $\phi : W_F \times SL_2(\CC) \rightarrow {^L}{M}$ be an elliptic tempered $L$-parameter.
Let $\sigma' \in \Pi_{\phi}(M')$ be given. 
Denote $I_1 = \{ 1, 2, \cdots, k \}$ if $m \geq 2$ and $c_0 \tau' \s \tau',$ 
otherwise, $I_1 = \{ 1 \leq i \leq k : n_i ~~ \text{is even}  \}.$ 
Denote $I_2 = \{1, 2, \cdots, k\} - I_1.$ 
Set
\begin{align*}
I_1(\sigma') &= \{ i \in I_1 : \sigma'_i \s  \widetilde{\sigma}'_i ~~\text{and}~~ \ii_{G'(n_i + m), ~ GL(n_i)\times G'(m)} (\sigma' \otimes \tau') ~~\text{is reducible}~  \},  \\
I_2(\sigma') &= \{ i \in I_2 : \sigma'_i \s  \widetilde{\sigma}'_i \}.
\end{align*}
Let $d_j$ denote the number of inequivalent $\sigma'_i$ such that $i \in I_j(\sigma'),$ and let $d=d_1+d_2.$ If $d_2=0,$ then
\[
R_{\sigma'} \s \ZZ_2^d.
\]
If $d_2 > 0,$ then
\[
R_{\sigma'} \s \ZZ_2^{d-1}.
\]
In either case, $R_{\sigma'}$ is a subgroup of $W_{M'}$ generated by sign changes. 
\end{cor}
The rest of Section \ref{main theorem} is devoted to the proof of Theorem \ref{iso bw R}.
\begin{lm} \label{identity of W}
Let $\phi : W_F \times SL_2(\CC) \rightarrow {^L}{M}$ be an elliptic tempered $L$-parameter.
Under the identity \eqref{identity for Weyls}, we have
\[
W(\sigma) = W_{\phi} = W(\sigma').
\]
\end{lm}
\begin{proof}
Let $w \in W_{\phi}$ be given. Then by abuse of notation its representative $w \in C_{\phi}(\widehat{G})$ satisfies
$^w \phi \s \phi.$ 
Set $\phi$ to be $\phi_1 \oplus \phi_2 \oplus \cdots \oplus \phi_k \oplus \phi_-.$
By the action of $W_{\widehat M}$ on $L$-parameters in Section \ref{action on parameters}, there exist $i, j$ with $1 \leq i, j \leq k$ (possibly, $i=j$) such that 
\begin{itemize}
 \item[(i)] $\phi_i \s \phi_j,$ or $\phi_i \s \tilde{\phi}_j$ if $\bG=SO_{2n+1}, Sp_{2n}$
 \item[(ii)] $\phi_i \s \phi_j,$ or $\phi_i \s \tilde{\phi}_j,$ or $\phi_i \s \tilde{\phi}_i$ and $c_0 \phi_- \s \phi_-$ if $\bG=SO_{2n}, SO^*_{2n}.$
\end{itemize}
Due to the local Langlands correspondence for $GL_n$ for the case $\phi_i,$ $\phi_j$ (\cite{he00, ht01}) and for $G$ for the case $\phi_-$ (\cite[Theorems 1.5.1]{art12} and Remark \ref{q-split and L-packet for SO}.2), we have $i, j$ with $1 \leq i, j \leq k$ (possibly, $i=j$) such that 
\begin{itemize}
 \item[(a)] $\sigma_i \s \sigma_j$ or $\sigma_i \s \widetilde{\sigma}_j$ if $\bG=SO_{2n+1}, Sp_{2n}$ (see \cite[Section 4]{goldberg-class}, \cite[Section 2]{herb93})
 \item[(b)] $\sigma_i \s \sigma_j,$ or $\sigma_i \s \widetilde{\sigma}_j,$ or $\sigma_i \s \tilde{\sigma}_i$ and $c_0 \sigma_- \s \sigma_-$ if $\bG=SO_{2n}, SO^*_{2n}$ (see \cite[Section 3]{herb93}, Appendix \ref{R-gps for SO*}).
\end{itemize}
By the action of $W_{M}$ on $\sigma = \sigma_1 \otimes \sigma_2 \otimes \cdots \otimes \sigma_k \otimes \tau$ in Section \ref{action on parameters}, we have 
\[
{^w}\sigma \s \sigma,
\]
which implies $w \in W(\sigma).$
Likewise, by the action of $W_{M'}$ on $\sigma'$ in Section \ref{WM-actions on Levis} we have 
\[
{^w}\sigma' \s \sigma',
\]
which implies $w \in W(\sigma').$ Note that, for the local Langlands correspondence of a non-split $F$-inner form $M',$ we apply \cite{hs11} and \cite[Theorems 1.5.1 and 9.4.1]{art12}. 

Let $w \in W(\sigma)$ be given. It then follows from the action of $W_{M}$ on $\sigma$ in Section \ref{WM-actions on Levis} that there exist $i, j$ with $1 \leq i, j \leq k$ (possibly, $i=j$) satisfying (a) and (b) above. 
Thus, by the local Langlands correspondence, we have (i) and (ii) above. The action of $W_{\widehat M}$ on $L$-parameters in Section \ref{action on parameters} then yields 
\[
{^w}\phi \s \phi,
\]
which implies $w \in C_{\phi}(\widehat G).$ Therefore, we have proved Lemma \ref{identity of W}.
\end{proof}
\begin{lm} \label{identity of W'}
Let $\phi : W_F \times SL_2(\CC) \rightarrow {^L}{M}$ be an elliptic tempered $L$-parameter.
Under the identity \eqref{identity for Weyls}, we then have
\[
W'(\sigma) = W^{\circ}_{\phi} = W'(\sigma').
\]
\end{lm}
\begin{proof}
Let us start with the case that $M$ is maximal.
Since $M$ is maximal, we have $W_M \s \ZZ/2\ZZ.$ Let $w \in W^{\circ}_{\phi}$ be given. 
If $w$  is a trivial element, it is clear that $w$ sits in $W' (\sigma)$ as well as $W' (\sigma')$. Suppose that $w$ is a non-trivial element. 
We then have $R_{\phi} = 1.$ 
From \eqref{useful exact} and \eqref{useful exact 2}, we have 
\[
S_{\phi}(\widehat G) = S_{\phi}(\widehat M) ~ ~ ~ \text{and} ~ ~ ~ S_{\phi, \scn}(\widehat G) = S_{\phi, \scn}(\widehat M).
\]
Thus, we have equalities of the sizes of $L$-packets
\[
\# \Pi_{\phi}(G)= \# \Pi_{\phi}(M)  ~ ~ ~ \text{and} ~ ~ ~ \# \Pi_{\phi}(G')= \# \Pi_{\phi}(M'),
\]
which imply that both induced representations $\ii_{G,M}(\sigma)$  and $\ii_{G',M'}(\sigma')$ are irreducible for all $\sigma \in \Pi_{\phi}(M)$ and $\sigma' \in \Pi_{\phi}(M'),$ respectively. Hence, we have $R_{\sigma}=R_{\sigma'}=1.$ So, $w$ must be in $W'(\sigma)$ and $W'(\sigma').$ 

Let $w \in W'(\sigma)$ be given. If $w$ is a trivial element, it is clear that $w$ sits in $W^{\circ}_{\phi}.$ Suppose $w$ is a non-trivial element. We then have $R_{\sigma} = 1.$ Thus, all induced representations $\ii_{G,M}(\sigma),~ \sigma \in \Pi_{\phi}(M),$ are irreducible, 
which implies $\# \Pi_{\phi}(G)= \# \Pi_{\phi}(M)$ and thus $R_{\phi}=1$ by \eqref{useful exact}. 
So, $w$ must be in $W^{\circ}_{\phi}.$ The same argument applies to $W'(\sigma').$ Therefore, we have proved Lemma \ref{identity of W'} for the maximal case $M$ and $M'.$

We return to an arbitrary $F$-Levi subgroup. We shall complete the proof by verifying two containments
\begin{equation} \label{containment 1}
W'(\sigma) \subset W^{\circ}_{\phi} 
\quad (W'(\sigma') \subset W^{\circ}_{\phi})
\end{equation}
and 
\begin{equation} \label{containment 2}
W'(\sigma) \supset W^{\circ}_{\phi} 
\quad (W'(\sigma') \supset W^{\circ}_{\phi}).
\end{equation}
To see \eqref{containment 1}, we let $w \in W'(\sigma)$ such that $w$ equals the element $w_\alpha$ representing the reflection with respect to some $\alpha \in \Phi(P,A_M).$ 
By definition, we have $\mu_{\alpha}(\sigma)=0.$   
We let $A_{\alpha}:=(\ker \alpha \cap A_M)^{\circ}$ be the identity component of $(\ker \alpha \cap A_M).$ 
Set $\bM_{\alpha}:=Z_{\bG}(A_{\alpha}),$ then $\bM_{\alpha}$ contains $\bM$ as a maximal $F$-Levi subgroup. From Remark \ref{remark for T}, we have
\[
T_{\phi} \subset \widehat M \subset \widehat M_{\alpha}.
\]
Set 
\begin{equation} \label{W alpha}
W_{\phi, \alpha}^{\circ} := N_{C_{\phi}(\widehat M_{\alpha})^{\circ}}(T_{\phi}) / Z_{C_{\phi}(\widehat M_{\alpha})^{\circ}}(T_{\phi}).
\end{equation}
By the above argument for the maximal case, we have 
\[
w \in W_{\phi, \, \alpha}^{\circ}.
\]
It follows that 
\[
w \in N_{C_{\phi}(\widehat M_{\alpha})^{\circ}}(T_{\phi}) \subset N_{C_{\phi}(\widehat G)^{\circ}}(T_{\phi})
\]
since  
\[
C_{\phi}(\widehat M_{\alpha})^{\circ} \subset C_{\phi}(\widehat G)^{\circ}.
\]
Thus, we must have $w \in W_{\phi}^{\circ}.$ The same argument applies to $W'(\sigma').$ 
By the definitions of $W'(\sigma)$ and $W'(\sigma'),$ we have proved the containment \eqref{containment 1}.

To see \eqref{containment 2}, we recall the argument \eqref{useful arthur's arg}. So, it remains to show that 
\begin{equation} \label{remains to show}
W'(\sigma) = W'(\sigma').
\end{equation}
We note that $\Phi(P, A_M)=\Phi(P', A_{M'})$ since $\bM$ and $\bM'$ are inner forms each other and their $F$-points are respectively of the form \eqref{form of M} and \eqref{form of M'}. 
For given $\alpha \in \Phi(P, A_M)=\Phi(P', A_{M'}),$ 
by the argument for the maximal case, we have 
\[
\mu_{\alpha} (\sigma) = 0 \quad \text{if and only if} \quad \mu_{\alpha} (\sigma') = 0
\]
since these Plancherel measures attached to $\alpha$ are those for the induced representations from maximal $F$-Levi subgroups.
By the definitions of $\Delta'_\sigma
$ and $\Delta'_{\sigma'}$ in Section \ref{section for def of R}, we then have 
\[
\Delta'_\sigma  = \Delta'_{\sigma'},
\]
which implies the containment \eqref{remains to show}.
Thus, we have proved the containment \eqref{containment 2}. This completes the proof of Lemma \ref{identity of W'}.
\end{proof} 
\begin{proof} [Proof of Theorem \ref{iso bw R}]
From Lemmas \ref{identity of W} and \ref{identity of W'}, we have 
\[
R_{\sigma} \s R_{\phi} \s R_{\sigma'}.
\]
Therefore, the definition \eqref{def of W_phi, sigma} yields that 
\begin{equation} \label{iso for max}
R_{\phi, \sigma} \s R_{\phi} \s R_{\phi, \sigma'}.
\end{equation}
\end{proof}
\section{Transfer of elliptic spectra} \label{transfer of elliptic spectra}
In this section, we show that the elliptic spectrum of $\bG=SO_{2n+1},$ $Sp_{2n},$ $SO_{2n},$ $SO^*_{2n}$ is identically transferred to its $F$--inner form $\bG'.$
An elliptic representation of $G$ is the one whose Harish-Chandra character (see \cite{hc81}) does not vanish on the elliptic regular set of $G,$ cf. \cite{art93}.
The elliptic spectrum is studied for $SO_{2n+1},$ $Sp_{2n},$ and $SO_{2n}$ by Herb \cite{herb93}, 
and for certain $F$-inner forms of $Sp_{4n}$ and $SO_{4n}$ by Hanzer \cite{han04}, 
in which both rely on the second named author's work in \cite{goldberg-class}. 
We start with the following result of Arthur for a general connected reductive group.

\begin{thm}(Arthur, \cite[Section 2.1]{art93}) \label{ell spec by arthur}
Let $\bG$ be a connected reductive group over $F$ and let $\bM$ be an $F$-Levi subgroup of $\bG.$ Let $\sigma \in \Pi_{\disc}(M)$ be given. 
Suppose that $R_{\sigma}$ is abelian and that $C(\sigma) \s \CC[R_{\sigma}].$ 
Then $\ii_{G,M}(\sigma)$ has an elliptic constituent if and only if all constituents of $\ii_{G,M}(\sigma)$ are elliptic, and  if and only if there is $w \in R$ such that $\mathfrak{a}_{w} = \mathfrak{z}.$ 
Here $\ma_{w} := \{ H \in \ma_M : wH=H \},$ and $\mathfrak{z}$ denotes the real Lie algebra of the split component $A_{\bG}$ of $\bG.$
\end{thm}

We recall Herb's results on the elliptic spectrum for $\bG=SO_{2n+1},$ $Sp_{2n},$ $SO_{2n}$ as follows.
\begin{pro}(Herb, \cite[Proposition 2.3 and 3.3]{herb93}) \label{split 2-cocycle by herb}
Let $\bG$ be $SO_{2n+1},$ $Sp_{2n}$ or $SO_{2n}$ over $F$ and let $\bM$ be an $F$-Levi subgroup of $\bG.$ Let $\sigma \in \Pi_{\disc}(M)$ be given. Then 
\[
C(\sigma) \s \CC[R_{\sigma}].
\]
\end{pro}
The result that $R_{\sigma}$ is abelian for $\bG$ \cite{goldberg-class} is combined with Theorem \ref{ell spec by arthur} and Proposition \ref{split 2-cocycle by herb} to give the elliptic spectrum for $\bG$ as follows.
\begin{thm} (Herb, \cite[Theorem 2.5]{herb93})
Let $\bG$ be $SO_{2n+1}$ or $Sp_{2n}$ over $F$ and let $\bM$ be an $F$-Levi subgroup of $\bG$ such that $M \s GL_{n_1}(F) \times GL_{n_2}(F) \times \cdots \times GL_{n_k}(F) \times G_{-}(m).$ Let $\sigma \in \Pi_{\disc}(M)$ be given. Then Then $\ii_{G,M}(\sigma)$ has elliptic constituent if and only if all constituents of $\ii_{G,M}(\sigma)$ are elliptic if and only if 
\[
R_{\sigma} \s \ZZ_2^k.
\]
\end{thm}
\begin{thm} (Herb, \cite[Theorem 3.5]{herb93})
Let $\bG$ be $SO_{2n}$ over $F$ and let $\bM$ be an $F$-Levi subgroup of $\bG$  of the form \eqref{form of M}.
Let $\sigma \in \Pi_{\disc}(M)$ be given. Then Then $\ii_{G,M}(\sigma)$ has elliptic constituent if and only if all constituents of $\ii_{G,M}(\sigma)$ are elliptic if and only if either
\[
R_{\sigma} \s \ZZ_2^k ~~ \text{or} ~~ R_{\sigma} \s (\ZZ_2)^{k-1} ~\text{with}~d_2~\text{even}.
\]
Here $d_2$ is defined unless $m\leq 1$ and $c_0 \tau \s \tau,$ in which case $d_2$ denotes the number of inequivalent $\sigma_i$ such that $\sigma_i \s \widetilde{\sigma}_i$ and $n_i$ is odd.
\end{thm}

Applying our results in Section \ref{main theorem} and the above arguments, we will present the elliptic spectrum for $\bG'.$ We will first deal with the case of an $F$--inner form $\bG'$ of $SO_{2n+1}$ or $Sp_{2n}.$
\begin{pro} \label{split 2-cocycle for G' bc}
Let $\bG'$ be an $F$--inner form of $SO_{2n+1}$ or $Sp_{2n},$ and let $\bM'$ be an $F$-Levi subgroup of $\bG'.$ Let $\sigma' \in \Pi_{\disc}(M')$ be given. Then 
\[
C(\sigma') \s \CC[R_{\sigma'}].
\]
\end{pro}
\begin{proof}
We follow the proof of \cite[Proposition 2.3]{herb93}. Recall from Corollary \ref{r-groups for G'=bc} that $C_1, \cdots, C_d$ are the generators of $R_{\sigma'}.$ For $1 \leq i \leq d,$ let $V_i$ be the representation space of $\sigma'_i.$ 
We note from Lemma \ref{conjugation lemma} that 
\[
\sigma'_i \s \widetilde{\sigma}'_i \s {^t}\bar{\sigma'}_i^{-1}.
\]
Let $T_i : V_i \rightarrow V_i$ be an intertwining operator between  $\sigma'_i$ and ${^t}\bar{\sigma}^{-1}.$ 
Again, from Lemma \ref{conjugation lemma}, the composite $T^2_i = T_i \circ T_i$ is a non-zero complex scalar. 
Thus we normalize $T_i$ so that $T^2_i =1.$ Now we extend $T_i$ to an endomorphism $T_i^V$ of the representation space $V$ of $\sigma' = \sigma'_1 \otimes \cdots \otimes \sigma'_k \otimes \tau'$ 
such that $T_i^V$ acts trivially on the space $V_j$ for all $1 \leq j \neq i \leq d.$ 
Then $T_i^V$ intertwines $C_i\sigma'$ and $\sigma'.$ 
Further,  we have $(T_i^V)^2 = 1$ from our normalization of $T_i.$ For  $1 \leq i \neq j \leq d,$ 
we also notice from the definition of $T_i^V$ that $T_i^V T_j^V = T_j^V T_i^V.$ Then, the map $C_i \mapsto T_i^V$ gives rise to a homomorphism from $\CC[R_{\sigma'}]$ to $\operatorname{End}_{G'}(\sigma').$ Thus, Proposition \ref{split 2-cocycle for G' bc} follows from Theorem \ref{thm for Knapp-Stein-Sil}.
\end{proof}
\begin{cor}
Let $G', M', \sigma'$ be as in Proposition \ref{split 2-cocycle for G' bc}. Then each constituent of $\ii_{G',M'}(\sigma')$ appears with multiplicity one.
\end{cor}
Combining Corollary \ref{r-groups for G'=bc}, Theorem \ref{ell spec by arthur}, and Proposition \ref{split 2-cocycle for G' bc}, we have the following description of the elliptic spectrum for an $F$--inner form $\bG'$ of $SO_{2n+1}$ or $Sp_{2n}.$
\begin{thm}
Let $\bG'$ be an $F$--inner form of $SO_{2n+1}$ or $Sp_{2n},$ and let $\bM'$ be an $F$-Levi subgroup of $\bG'$ of the form \ref{form of M'}
Let $\sigma' \in \Pi_{\disc}(M')$ be given. Then, $\ii_{G',M'}(\sigma')$ has elliptic constituent if and only if all constituents of $\ii_{G',M'}(\sigma')$ are elliptic if and only if 
\[
R_{\sigma'} \s \ZZ_2^k.
\]\qed
\end{thm}

Now, we will deal with the case of an $F$--inner form $\bG'$ of $SO_{2n}$ or $SO^*_{2n}.$
\begin{pro}  \label{split 2-cocycle for G' d}
Let $\bG'$ be an $F$--inner form of $SO_{2n}$ or $SO^*_{2n},$ and let $\bM'$ be an $F$-Levi subgroup of $\bG'.$ Let $\sigma' \in \Pi_{\disc}(M')$ be given. Then 
\[
C(\sigma') \s \CC[R_{\sigma'}].
\]
\end{pro}
\begin{proof}
We follows the proof of \cite[Proposition 3.3]{herb93}. Suppose that $m=0,$ or that $m \geq 2$ but $c_0\tau' \not\s \tau'.$ In this case, $I_1$ (respectively, $I_2$) consist of all $i$'s such that $n_i$ is even (respectively, odd), cf. Remark \ref{remark on c_0}. 
If $d_2 \leq 1,$ then $R_{\sigma'}$ is generated by $d_1$ sign changes in indices $i \in I_1(\sigma'),$ and the proof is the same as that of Proposition \ref{split 2-cocycle for G' bc}. Assume that $d_2 \geq 2.$ 
Then, Corollary \ref{r-groups for G'=d} tells us that $R_{\sigma'} \s \ZZ_2^{d_1+d_2 -1}.$ 
Note that the set
\[
\{ 
C_1C_{d_2}, C_2C_{d_2}, \cdots, C_{{d_2}-1}C_{d_2}, C_{{d_2}+1}, \cdots, C_d
\}
\]
forms a complete set of generators for $R_{\sigma'}.$ 
As in the proof of Proposition \ref{split 2-cocycle for G' bc}, for $1\leq i \leq d,$ we obtain $T_i:V_i\rightarrow V_i$ intertwining $\sigma'_i$ and $\widetilde\sigma'_i,$ such that $T_i^2 = 1$ and such that each $T_i$ is extended to an endomorphism $T_i^V$ of the representation space $V$ of $\sigma'.$ 
Again, we note that $(T_i^V)^2=1$ and $T_i^VT_j^V=T_j^VT_i^V.$ Then, the map 
\begin{align*}
C_iC_{d_2} \mapsto T_i^VT_{d_2}^V, \quad & \text{if} ~~~1 \leq i \leq d_2-1 \\
C_i \mapsto T_i^V, \quad & \text{if} ~~~d_2 + 1 \leq i \leq d
\end{align*}
defines a homomorphism from $\CC[R_{\sigma'}]$ to $\operatorname{End}_{G'}(\sigma').$ 

Suppose $m \geq 2$ and $c_0\tau' \s \tau'.$ Then, Corollary \ref{r-groups for G'=d} tells us that $d_1=d$ and $R_{\sigma'} \s \ZZ_2^{d}.$ 
Note that $C_1, \cdots, C_{d}$ form a complete set of generators for $R_{\sigma'}.$ We renumber indices so that 
\begin{align*}
I_1(\sigma') \cap \{i : n_i ~~\text{odd} \} &= \{ 1, \cdots, p \} \\
I_1(\sigma') \cap \{i : n_i ~~\text{even} \} & = \{ 
p+1, \cdots, d \}.
\end{align*}
From the condition $c_0\tau' \s \tau',$ we choose an intertwining operator $T_{-}:V_{-} \rightarrow V_{-},$ where $V_-$ is the representation space of $\tau'.$ 
As in the proof of Proposition \ref{split 2-cocycle for G' bc}, we normalize $T_-$ so that $T_-^2=1,$ and we extend it to an operator $T_-^V$ on $V$ such that $T_-^V$ acts trivially on the space $V_i$ for all $1 \leq i \leq d.$
Again, the map 
\begin{align*}
C_i \mapsto T_i^VT_-^V, \quad & \text{if} ~~~ 1 \leq i \leq p \\
C_i \mapsto T_i^V, \quad & \text{if} ~~~ p + 1 \leq i \leq d
\end{align*}
defines a homomorphism from $\CC[R_{\sigma'}]$ to $\operatorname{End}_{G'}(\sigma').$ Thus, Proposition \ref{split 2-cocycle for G' d} follows from Theorem \ref{thm for Knapp-Stein-Sil}.
\end{proof}
\begin{cor}
Let $G', M', \sigma'$ be as in Proposition \ref{split 2-cocycle for G' d}. Then each constituent of $\ii_{G',M'}(\sigma')$ appears with multiplicity one.
\end{cor}
Corollary \ref{r-groups for G'=d}, Theorem \ref{ell spec by arthur}, and Proposition \ref{split 2-cocycle for G' d} gives the following description of the elliptic spectrum for an $F$--inner form $\bG'$ of $SO_{2n}$ or $SO^*_{2n}.$
\begin{thm} 
Let $\bG'$ be an $F$--inner form of $SO_{2n}$ or $SO^*_{2n},$ and let $\bM'$ be an $F$-Levi subgroup of $\bG'$ of the form \eqref{form of M'}.
Let $\sigma' \in \Pi_{\disc}(M')$ be given. Then, $\ii_{G',M'}(\sigma')$ has elliptic constituent if and only if all constituents of $\ii_{G',M'}(\sigma')$ are elliptic if and only if either
\[
R_{\sigma'} \s \ZZ_2^k ~~ \text{or} ~~ R_{\sigma'} \s (\ZZ_2)^{k-1} ~\text{with}~\text{even}~d_2.
\]
Here $d_2$ is defined unless $m \leq 1$ and $c_0 \tau' \s \tau',$ in which case $d_2$ denotes the number of inequivalent $\sigma'_i$ such that $\sigma'_i \s \widetilde{\sigma'}_i$ and $n_i$ is odd.
\end{thm}
\appendix

\section{$R$-groups for a quasi-split even orthogonal group}  \label{R-gps for SO*}

We give a brief description of the Knapp-Stein $R$--groups for the non-split quasi-split classical group $SO_{2n}^*.$  In short, the arguments used for the split case are valid here as well.  
Let $\bG=\bG(n)=SO_{2n}^*$ be defined over $F,$  and $G=\bG(F).$  We also let $\tilde\bG=\tilde\bG(n)=O_{2n}^*.$  Let $\bold T$ be the maximal torus of diagonal elements and $\bold T_d$ the maximal split subtorus, as described in \cite{gs98}. 
Let $\bP=\bM\bN$ be  a standard parabolic subgroup with Levi component $\bM$ and unipotent radical $\bN.$ Then
\[
\bM\simeq GL_{n_1}\times GL_{n_2}\times \cdots\times GL_{n_k}\times\bG(m),
\]
where $n_1+n_2+\cdots +n_k+m=n,$ and $m\geq 1.$ Let $\bA=\bA_\bM$ be the split component of $\bM.$  Fix an element $c_0\in O^*_{2m}\setminus SO_{2m}^*.$
Then $W(\bG,\bA)\simeq S\rtimes\ZZ_2^k,$
where  $S\subset S_k.$
More precisely, since $m>0,$ we have
$S=\left\langle (ij)|n_i=n_j\right\rangle.$ The (block) sign changes are given by $\sC=\left\langle C_i|n_i\text{ is even }\right\rangle\,\times\,\left\langle C_ic_0|n_i,n_j\text{ are odd}\right\rangle.$
With this in mind we define
\[
\bar C_i=\begin{cases} C_i&\text{if } n_i \text{ is even;}\\C_ic_0&\text{if }n_i\text{ is odd.}\end{cases}
\]

We assume, without loss of generality, that $n_1,\dots,n_t$ are all even, and $n_{t+1},\dots,n_k$ are odd.
For $1\leq i\leq k,$ we set $\ds{m_i=\sum\limits_{j=1}^i n_j.}$
We have the reduced root system, 
\[
\Phi(\bP,\bA_\bM)=\left\{e_{m_i}\pm e_{m_j}\mid 1\leq i<j\leq k\right\}\cup\left\{e_{m_i}+e_n\mid 1\leq i\leq k\right\}.
\]
We let $\alpha_{ij}=e_{m_i}-e_{m_j},$ $\beta_{ij}=e_{m_i}+e_{m_j},$ and $\gamma_i=e_{m_i}+e_n.$

Let $\sigma\in\Pi_{\disc}(M).$  Then we have
\[
\sigma\simeq\sigma_1\otimes\sigma_2\otimes\cdots\otimes\sigma_k\otimes\tau,
\]
with  each $\sigma_i\in\Pi_{\disc}(GL_{n_i}(F)),$ and $\tau\in\Pi_{\disc}(G(m)).$

\begin{lm}
For $1\leq i<j\leq k,$ we have
$\mu_{\a_{ij}}(\sigma)=0$ if and only if $\sigma_i\simeq\sigma_j,$ and $\mu_{\b_{ij}}(\sigma)=0$ if and only if $\sigma_i\simeq\tilde\sigma_j.$
\end{lm}
\begin{proof}
As in the cases of all split classical groups this follows from the result of Bernstein and Zelevinski \cite{bz77} and Olsanskii \cite{ol74}.
\end{proof}

\begin{lm}
Suppose $w\in R_\sigma$ and $w=sc,$ with $s\in S$ and $c\in\sC.$  Then $w=1.$
\end{lm}
\begin{proof}
We assume the cycle $(12\dots j)$ appears in the disjoint cycle decomposition of $s.$ By conjugation, we may assume $c$ changes at most two of the block signs among $1,2,\dots, j.$  That is, we may assume $1,$ $C_j,$ or $C_{j-1}C_j$ are all the sign changes  $C_\ell$ appearing in $c$ with $1\leq \ell\leq j.$
First, suppose $c$ changes no signs among $1,2,,\dots,j.$  Then $w\sigma\simeq\sigma$ implies $\sigma_1\simeq\sigma_2\simeq\dots\simeq\sigma_j.$  By the lemma, $\a_{1j}\in\Delta'_{\sigma}$ and $w\a_{1j}=-\a_{12}<0.$  This contradicts $w\in R_\sigma .$  If $c$ changes just the sign of $j,$ then we have $\sigma_1\simeq\sigma_2\simeq\dots\simeq\sigma_j\simeq\tilde \sigma_1.$  Since $\sigma_1\simeq\tilde\sigma_j,$ we have $\b_{1j}\in\Delta',$ and $w\b_{1j}=\a_{12}<0,$ again contradicting $w\in R_\sigma .$  Finally, if $c$ changes the sign of $j-1$ and $j,$ then, again $\sigma_1\simeq\tilde\sigma_j,$ and this $\b_{ij}\in\Delta'.$  Now $w\b_{1j}=-\b_{12}<0,$ so again, contradicting $w\in R_\sigma .$  Thus, $s=1.$
\end{proof}

\begin{cor}
The $R$--group $R_\sigma$ is an elementary $2$--group.
\end{cor}

If $c_0\tau\not\simeq\tau,$ then we set $d_1$  to be the number of inequivalent classes among $\sigma_1,\dots,\sigma_t$ such that
\[
i_{G(m+n_i), GL_{n_i}(F)\times G(m)}(\sigma_i\otimes\tau)
\]
is reducible, and set $d_2$ to be the number of inequivalent classes among $\sigma_{t+1},\dots,\sigma_k$ satisfying $\sigma_i\simeq\tilde\sigma_i.$  We then let 
\[
d=\begin{cases} d_1+d_2-1,&\text{ if }d_2>0;\\d_1&\text{otherwise}.\end{cases}
\]
On the other hand, if $c_0\tau\simeq\tau,$ we set $d$ to be the number of inequivalent classes among $\sigma_1,\dots,\sigma_k$ such that 
$i_{G(m+n_i), ~ GL_{n_i}(F)\times G(m)}(\sigma_i\otimes\tau)$ is reducible.

We then have the following result.
\begin{thm}
Let $G,$ $M,$ $\sigma$  and  $d$ be as above.  Then 
$R_\sigma \simeq\ZZ_2^d.$ 
Moreover, we can give a precise description of the $R$--group which mirrors that for the split groups $SO_{2n}(F).$
\end{thm}

The proof can be given verbatim as in \cite{goldberg-class}, but is simpler if we follow the proof of \cite{ban-goldberg-gspin}.



\begin{thebibliography}{CPSS11}

\bibitem[AJ12]{adams-vogan}
Jeffrey Adams and David A.~Vogan Jr, \emph{The contragredient}, arXiv:1201.0496.

\bibitem[Art89]{art89ast}
James Arthur, \emph{Unipotent automorphic representations: conjectures},
  Ast\'erisque (1989), no.~171-172, 13--71, Orbites unipotentes et
  repr{\'e}sentations, II. 

\bibitem[Art93]{art93}
\bysame, \emph{On elliptic tempered characters}, Acta Math. \textbf{171}
  (1993), no.~1, 73--138.

\bibitem[Art12]{art12}
\bysame, \emph{The endoscopic classification of representations: Orthogonal and
  symplectic groups}, to appear as a {C}olloquium {P}ublication of the
  {A}{M}{S} (2012).

\bibitem[BG12]{ban-goldberg-gspin} D.~Ban and D.~Goldberg, \emph{$R$--groups, elliptic representations, and parameters for  $GSpin$ groups}, submitted

\bibitem[BZ77]{bz77} I. N. Bernstein and A. V. Zelevinsky, \emph{Induced representations of reductive $p$--adic groups. I}, Ann. Sci. Ecole   Norm. Sup. (4) \textbf{10} (1977), 441-472.

\bibitem[Bor79]{bo79}
A.~Borel, \emph{Automorphic {$L$}-functions}, Proc. Sympos. Pure Math., XXXIII,
  Amer. Math. Soc., Providence, R.I., 1979, pp.~27--61. 

\bibitem[Bor91]{borel91}
Armand Borel, \emph{Linear algebraic groups}, second ed., Graduate Texts in
  Mathematics, vol. 126, Springer-Verlag, New York, 1991. 

\bibitem[Cho13]{choiy}
Kwangho Choiy, \emph{Transfer of {P}lancherel measures for unitary supercuspidal
  representations between {$p$}-adic inner forms}, to appear in Canadian
  Journal of Mathematics (Published electronically on February 21, 2013).
  
\bibitem[CG12]{goldberg-choiy-1}
Kwangho Choiy and David Goldberg, \emph{{$R$}-groups for {$p$}-adic inner forms of {$SL_n$}}, submitted; arXiv:1211.5054.

\bibitem[CL13]{chao-li}
Kuo~Fai Chao and Wen-Wei Li, \emph{Dual {$R$}--groups of the inner forms of
  {$SL(N)$}}, to appear in Pacific J. Math.; arXiv:1204.0132.

\bibitem[CPSS11]{cpssh11}
J.~W. Cogdell, I.~I. Piatetski-Shapiro, and F.~Shahidi, \emph{Functoriality for
  the quasisplit classical groups}, On certain {$L$}-functions, Clay Math.
  Proc., vol.~13, Amer. Math. Soc., Providence, RI, 2011, pp.~117--140.
  

\bibitem[DKV84]{dkv}
P.~Deligne, D.~Kazhdan, and M.-F. Vign{\'e}ras, \emph{Repr\'esentations des
  alg\`ebres centrales simples {$p$}-adiques}, Travaux en Cours, Hermann,
  Paris, 1984, pp.~33--117.

\bibitem[Gol94]{goldberg-class}
David Goldberg, \emph{Reducibility of induced representations for {${\rm
  Sp}(2n)$} and {${\rm SO}(n)$}}, Amer. J. Math. \textbf{116} (1994), no.~5,
  1101--1151. 
  
  \bibitem[GS98]{gs98} D.~Goldberg and F.~Shahidi\emph{On the Tempered Spectrum of Quasi-split Classical Groups},  Duke J.
Math., \textbf{92}; 255-294, 1998

\bibitem[Han04]{han04}
Marcela Hanzer, \emph{{$R$} groups for quaternionic {H}ermitian groups}, Glas.
  Mat. Ser. III \textbf{39(59)} (2004), no.~1, 31--48. 
  
\bibitem[HC81]{hc81}
Harish-Chandra, \emph{A submersion principle and its applications}, Proc.
  Indian Acad. Sci. Math. Sci. \textbf{90} (1981), no.~2, 95--102. 

\bibitem[Hen00]{he00}
Guy Henniart, \emph{Une preuve simple des conjectures de {L}anglands pour
  {${\rm GL}(n)$} sur un corps {$p$}-adique}, Invent. Math. \textbf{139}
  (2000), no.~2, 439--455. 

\bibitem[Her93]{herb93}
Rebecca~A. Herb, \emph{Elliptic representations for {${\rm Sp}(2n)$} and {${\rm
  SO}(n)$}}, Pacific J. Math. \textbf{161} (1993), no.~2, 347--358. 

\bibitem[HS11]{hs11}
Kaoru Hiraga and Hiroshi Saito, \emph{On {L}-packets for inner forms of {${\rm
  SL}(n)$}}, Mem. Amer. Math. Soc. \textbf{215} (2011), 97.

\bibitem[HT01]{ht01}
Michael Harris and Richard Taylor, \emph{The geometry and cohomology of some
  simple {S}himura varieties}, Annals of Mathematics Studies, vol. 151,
  Princeton University Press, Princeton, NJ, 2001, With an appendix by Vladimir
  G. Berkovich. 

\bibitem[JS04]{js04}
Dihua Jiang and David Soudry, \emph{Generic representations and local
  {L}anglands reciprocity law for {$p$}-adic {${\rm SO}_{2n+1}$}},
  Contributions to automorphic forms, geometry, and number theory, Johns
  Hopkins Univ. Press, Baltimore, MD, 2004, pp.~457--519.
  
\bibitem[Kot82]{kot82}
Robert~E. Kottwitz, \emph{Rational conjugacy classes in reductive groups}, Duke
  Math. J. \textbf{49} (1982), no.~4, 785--806.

\bibitem[Kot84]{kot84}
\bysame, \emph{Stable trace formula: cuspidal tempered terms}, Duke Math. J.
  \textbf{51} (1984), no.~3, 611--650.

\bibitem[Kot86]{kot86}
\bysame, \emph{Stable trace formula: elliptic singular terms}, Math. Ann.
  \textbf{275} (1986), no.~3, 365--399.

\bibitem[Kot97]{kot97}
\bysame, \emph{Isocrystals with additional structure. {II}}, Compositio Math.
  \textbf{109} (1997), no.~3, 255--339.
    

\bibitem[KS72]{ks72}
A.~W. Knapp and E.~M. Stein, \emph{Irreducibility theorems for the principal
  series}, Conference on {H}armonic {A}nalysis ({U}niv. {M}aryland, {C}ollege
  {P}ark, {M}d., 1971), Springer, Berlin, 1972, pp.~197--214. Lecture Notes in
  Math., Vol. 266. 

\bibitem[MS00]{ms00}
Goran Mui{\'c} and Gordan Savin, \emph{Complementary series for {H}ermitian
  quaternionic groups}, Canad. Math. Bull. \textbf{43} (2000), no.~1, 90--99.
 

\bibitem[MW03]{mw03}
Colette M{\oe}glin and Jean-Loup Waldspurger, \emph{Paquets stables de
  repr\'esentations temp\'er\'ees et de r\'eduction unipotente pour {${\rm
  SO}(2n+1)$}}, Invent. Math. \textbf{152} (2003), no.~3, 461--623. 

\bibitem[O74]{ol74} G. I. Ol'sanski\v i, \emph{Intertwining operators and complementary series in the class of representations induced from parabolic subgroups of the general linear group over a locally compact division algebra}, Math. USSR-Sb.\textbf{22} (1974), 217-254.

\bibitem[PR94]{pr94}
Vladimir Platonov and Andrei Rapinchuk, \emph{Algebraic groups and number
  theory}, Pure and Applied Mathematics, vol. 139, Academic Press Inc., Boston,
  MA, 1994. 

\bibitem[Sat71]{sa71}
I.~Satake, \emph{Classification theory of semi-simple algebraic groups}, New York, 1971, With an appendix by M. Sugiura, Notes prepared
  by Doris Schattschneider, Lecture Notes in Pure and Applied Mathematics, 3.
  

\bibitem[Sil78]{sil78}
Allan~J. Silberger, \emph{The {K}napp-{S}tein dimension theorem for {$p$}-adic
  groups}, Proc. Amer. Math. Soc. \textbf{68} (1978), no.~2, 243--246.
  

\bibitem[Sil79]{sil78cor}
\bysame, \emph{Correction: ``{T}he {K}napp-{S}tein dimension theorem for
  {$p$}-adic groups''\ [{P}roc. {A}mer. {M}ath. {S}oc. {\bf 68} (1978), no. 2,
  243--246;\ {MR} {\bf 58} \#11245]}, Proc. Amer. Math. Soc. \textbf{76}
  (1979), no.~1, 169--170. 
  
\bibitem[Spr98]{springer98}
T.~A. Springer, \emph{Linear algebraic groups}, second ed., Progress in
  Mathematics, vol.~9, Birkh\"auser Boston Inc., Boston, MA, 1998. 

\end{thebibliography}
\end{document}